\documentclass[smallextended]{svjour3}

\usepackage{amssymb}
\usepackage{amsmath}
\usepackage{amsfonts}
\usepackage{graphicx}
\usepackage{bm}
\usepackage{mathtools}
\mathtoolsset{showonlyrefs=true}
\usepackage{tabls}
\usepackage{multirow}
\usepackage[compatibility=false]{caption}
\usepackage{subcaption}

\newcommand{\cC}{\mathcal{C}}
\newcommand{\cT}{\mathcal{T}}
\newcommand{\cL}{\mathcal{L}}
\newcommand{\cP}{\mathcal{P}}
\newcommand{\bip}{\mathcal{BIP}}
\newcommand{\bC}{\mathbb{C}}
\newcommand{\bR}{\mathbb{R}}
\newcommand{\bN}{\mathbb{N}}
\newcommand{\mb}{\mbox{}}
\newcommand{\Cmr}{C_{\mathrm{MR}}}
\newcommand{\Cdmr}{C_{\mathrm{DMR}}}
\newcommand{\inj}{\hookrightarrow}
\DeclareMathOperator{\ext}{ext}
\newcommand{\relmiddle}[1]{\mathrel{}\middle#1\mathrel{}}
\newcommand{\lmid}{\relmiddle{|}}

\newtheorem{thm}{Theorem}

\spnewtheorem{assum}{Assumption}{\bfseries}{\upshape}

\title{Discrete maximal regularity and the finite element method for parabolic equations}
\titlerunning{Discrete maximal regularity and FEM for parabolic equations}
\author{T. Kemmochi \and N. Saito}

\institute{
T. Kemmochi
\at Graduate School of Mathematical Sciences, The University of Tokyo, Komaba 3-8-1, Meguro-ku, Tokyo 153-8914, Japan \\
Tel.: +81-3-5465-7005,
Fax: +81-3-5465-7015,
\email{kemmochi@ms.u-tokyo.ac.jp}
\and
N. Saito
\at Graduate School of Mathematical Sciences, The University of Tokyo, Komaba 3-8-1, Meguro-ku, Tokyo 153-8914, Japan \\
\email{norikazu@ms.u-tokyo.ac.jp}
}

\begin{document}

\maketitle

\begin{abstract}
Maximal regularity is a fundamental concept in the theory of partial differential equations. In this paper, we establish a
fully discrete version of maximal regularity for a parabolic
equation. We derive various stability results in 
$L^p(0,T;L^q(\Omega))$ norm, $p,q\in (1,\infty)$ for the finite
element approximation with the mass-lumping to the linear heat
equation. Our method of analysis is an operator theoretical one using
pure imaginary powers of operators and might be a discrete version of G.~Dore and A.~Venni (On the closedness of the sum of two closed operators. \emph{Math.\ Z.}, 196(2):189--201, 1987).
As an application,
optimal order error estimates in that norm are proved. Furthermore,
we study the finite element approximation for semilinear heat equations
with locally Lipschitz continuous nonlinearity and offer a new
method for deriving optimal order error estimates. Some interesting auxiliary
results including discrete Gagliardo-Nirenberg and Sobolev inequalities
are also presented.   

\keywords{maximal regularity \and
parabolic equation \and
finite element method}
\subclass{35K91 \and 65M60}
\end{abstract}

\section{Introduction}
\label{sec:intro}

Let $\Omega$ be a bounded domain in $\mathbb{R}^d$, $d=2,3$, with the boundary
$\partial\Omega$. Let $J_T = (0,T)$ be a time interval with
$T\in (0,\infty]$. 
We consider the finite element approximation of linear heat equation for
the function $u=u(x,t)$ of $(x,t)\in\overline{\Omega}\times [0,T)$: 
\begin{equation}
\begin{cases}
\partial_t u = \Delta u + g  &\text{in } \Omega \times  J_T, \\
u = 0                             &\text{on }\partial\Omega\times J_T, \\
u|_{t=0} = u_0                 &\text{on }\Omega,
\end{cases}
\label{eq:heat}
\end{equation}
where $\partial_t u=\partial u/\partial t$, 
$\Delta u=\sum_{j=1}^d \partial^2 u/\partial {x_j}^2$, 
$g=g(x,t)$, and $u_0=u_0(x)$; $g$ and $u_0$ are prescribed functions. All functions and function
spaces considered in
this paper are complex-valued.  

The purpose of this paper is to derive
various stability estimates in the $L^p(J_T;L^q(\Omega))$ norm
\begin{equation*}
  \|v\|_{L^p(J_T;L^q(\Omega))}=\left[ \int_{0}^T
\left(\int_\Omega |v(x,t)|^q~dx\right)^{1/q}~dt\right]^{1/p}
\end{equation*}
and discrete $L^p(J_T;L^q(\Omega))$ norm defined as \eqref{eq:dlplq}
with $X=L^q(\Omega)$,
where $p,q\in (1,\infty)$.
As applications of those estimates, we also
derive optimal order error estimates in those norms for the finite
element approximations of \eqref{eq:heat} and semilinear heat equation
\begin{equation}
\begin{cases}
\partial_t u = \Delta u + f(u) & \text{in } \Omega \times J_T \\
u = 0 & \text{on } \partial \Omega \times J_T, \\
u|_{t=0} = u_0 & \text{in } \Omega,
\end{cases}
\label{eq:semilinear}
\end{equation}
where $f:\mathbb{C}\to\mathbb{C}$ is a prescribed
function. Particularly, we assume only a locally Lipschitz continuity
and offer a new method of error analysis for \eqref{eq:semilinear}.

In other words, we intend to develop
a discrete version of theory of maximal regularity for evolution
equations of parabolic type. 
To recall maximal regularity in a general context, 
let us consider an abstract Cauchy problem on a Banach space $X$ as
\begin{equation}
\begin{cases}
u'(t) = Au(t) + g(t), & t \in J_T, \\
u(0) = 0,
\end{cases}
\label{eq:acp}
\end{equation}
where $A$ is a densely defined closed operator on $X$ with the
domain $D(A)\subset X$, 
$g\colon J_T \to X$ is a given function,
$u \colon J_T \to X$ is an unknown function and $u'(t)=du(t)/dt$. 

\begin{definition}[Maximal regularity, MR, CMR]
Let $p \in (1, \infty)$.
The operator $A$ has \emph{maximal $L^p$-regularity ($L^p$-MR)} on $J_T$, if and
 only if, for every $g \in L^p(J_T; X)$, there exists a unique solution 
$u \in W^{1,p}(J_T; X) \cap L^p(J_T; D(A))$
of \eqref{eq:acp} satisfying
\begin{equation}
\| u \|_{L^p(J_T; X)} + \| u' \|_{L^p(J_T; X)} + \| Au \|_{L^p(J_T; X)}
\le \Cmr \| g \|_{L^p(J_T; X)},
\label{eq:mr}
\end{equation}
where $\Cmr > 0$ denotes a constant that is independent of $g$.
We say that $A$ has \emph{maximal regularity (MR)} if $A$ has maximal
 $L^p$-regularity for some $p \in (1, \infty)$ (see Lemma \ref{lem:LpMR}). 
To distinguish $L^p$-MR and MR from the discrete versions introduced later, we say
 that $A$ has \emph{continuous maximal $L^p$-regularity ($L^p$-CMR)} and \emph{continuous maximal regularity (CMR)} .
\label{def:mr1}
\end{definition}

It is proved that the $L^q(\Omega)$ realization $A$ of
$\Delta$ with $D(A)=W^{2,q}(\Omega)\cap
W^{1,q}_0(\Omega)$ has $L^p$-CMR for any $p,q\in
(1,\infty)$ ({see \cite{DenHP03,KunW04}}). The problem \eqref{eq:heat}
admits a unique solution $u\in W^{1,p}(J_T; L^q(\Omega)) \cap L^p(J_T; D(A))$
satisfying \eqref{eq:mr} with $u_0=0$.
This result implies that $\partial_tu$ and $\Delta u$ are
well defined and have the same regularity as the right-hand side
function $g$. Moreover, $\partial_tu$ and $\Delta u$ cannot be in a
better function space than $g$, since $g=\partial_tu-\Delta u$. This
is not a trivial fact. For comparison, we recall the solution obtained
using the analytical
semigroup theory, which is a powerful method to establish the well-posedness of
\eqref{eq:heat} and \eqref{eq:semilinear}.
For example, assume $g\in
C^\sigma(\overline{J_T};L^q(\Omega))$ for some $\sigma\in (0,1)$, that
is, assume 
\begin{equation*}
\sup_{t,s\in \overline{J_T},~t\ne s} \frac{\|g(t)-g(s)\|_{L^q(\Omega)}}{|t-s|^\sigma}<\infty.
\end{equation*}
Then, by application of the analytical semigroup theory, we can prove that the
problem \eqref{eq:heat} with $u_0=0$ admits a unique solution
$u\in C(\overline{J_T};X)\cap C(J_T;D(A))\cap C^1(J_T;L^q(\Omega))$; see \cite[Theorems 4.3.2, 7.3.5]{Paz83}. 
However, we are able to obtain slightly less regularity
$\partial_tu-\Delta u\in
C(J_T;L^q(\Omega))$ than $g$. To obtain the same regularity
$\partial_tu-\Delta u\in
C^\sigma(\overline{J_T};L^q(\Omega))$, we must further assume
$g(x,0)=0$ for all $x\in\Omega$; see \cite[Theorem 4.3.5]{Paz83}.
Therefore, $W^{1,p}(J_T; L^q(\Omega)) \cap L^p(J_T; D(A))$ is an
appropriate function space to study parabolic
equations such as \eqref{eq:heat}.
Moreover, CMR is a ``stronger'' property than the
generation of analytical
semigroup in the sense that, if $A$ has CMR, then $A$
generates the analytical (bounded) semigroup (cf.~\cite{Dor93}). 
Although CMR is a concept for linear equations, it actually has many
important applications to nonlinear equations, as reported in the literature \cite{Ama05,KunW04,Shi12}.
Moreover, the analytic semigroup theory and its discrete counterparts
play important roles in construction and study of numerical schemes for
parabolic equations
(see e.g.~\cite{EriJL98,FujSS01,GavM05,Sai07,Sai12,ZhoS16}).
Therefore, it is natural to wonder whether a discrete
version of CMR is available.

This study has another motivation. Considering the problem \eqref{eq:semilinear} with $f(u)=u|u|^{\alpha}$ for $\alpha>0$, then
without loss of the generality, we assume $0\in\Omega$. 
Let $\lambda>0$. Then the function 
\[
 u_\lambda(x,t)=\lambda^{\frac{2}{\alpha}}u(\lambda x,\lambda^2t)
\]
also solves \eqref{eq:semilinear} where $\Omega$ and $J_T$ are replaced,
respectively, by
$\Omega_\lambda =\{\lambda^{-1}x\mid x\in\Omega\}$ and $J_{T/\lambda^2}$. Moreover,
if $p,q\in (1,\infty)$ satisfy 
\begin{equation}
\label{eq:si}
 \frac{2}{\alpha}=\frac{d}{p}+\frac{2}{q},
\end{equation}
we have
\[
 \|u_\lambda\|_{L^p(J_{T/\lambda^{2}};L^q(\Omega_\lambda))}=\|u\|_{L^p(J_{T};L^q(\Omega))}
\]
for any $\lambda>0$. 
Those $p,q$ are called the scale invariant exponents. The function 
space $L^p(J_{T};L^q(\Omega))$ with $p,q$ satisfying \eqref{eq:si} plays
a crucially important role in the study of time-local and time-global well-posedness of
\eqref{eq:semilinear}. Furthermore, such a scaling argument is applied
to deduce a novel numerical method for solving \eqref{eq:semilinear}
(see \cite{BerK88}). Therefore, it would be interesting to
derive stability and error estimates in those norms from the dual
perspectives of numerical and theoretical analysis.

Based on those motivations, we studied a time discrete version of maximal
regularity for \eqref{eq:acp} in an earlier study \cite{Kem15}. Let
\begin{equation}
 \label{eq:NT}
       N_T=\begin{cases}
	    \lfloor T/\tau \rfloor & (T<\infty)\\
	    N_\infty=\infty & (T=\infty) .
	    \end{cases}
\end{equation}
We consider the implicit $\theta$ scheme for \eqref{eq:acp} given as 
\begin{equation}
\begin{dcases}
\frac{u^{n+1} - u^n}{\tau} = Au^{n+\theta} + g^{n+\theta}, & n = 0,1,\dots, N_T -1, \\
u^0 = 0,
\end{dcases}
\label{eq:dcp}
\end{equation}
where $\tau > 0$ is the time increment, $\theta \in [0,1]$, $g = (g^n)_{n=0}^{N_T}$ is a given $X^{N_T+1}$-valued function, 
and $u = (u^n)_{n=0}^{N_T}$ is an unknown $X^{N_T+1}$-valued function. Set 
\begin{equation*}
v^{n+\theta} = (1-\theta)v^n + \theta v^{n+1}
\end{equation*}
for a sequence $v = (v^n)_n$. 
We moreover assume that $A$ is bounded when $\theta \ne 1$. 
The function $u^n$ might be an approximation of $u(n\tau)$ for
$n=1,\dots, N_T$.

We introduce the space $l^p(N;X)$ by setting 
\begin{equation*}
l^p(N;X) = \begin{cases}
X^{N+1}, & N \in \bN, \\
l^p(\bN; X), & N = \infty,
\end{cases}
\end{equation*}
and let 
\begin{gather}
 \| v \|_{l^p_\tau(N;X)} = \left(\sum_{n=0}^{N-1} \| v^n \|_X^p \tau\right)^{1/p} , \label{eq:dlplq}\\
D_\tau v = \left( \frac{v^{n+1} - v^n}{\tau} \right)_{n=0}^{N-1}, \qquad
 Av = (Av^n)_{n=0}^N ,\label{eq:d1}\\
 v_\theta = (v^{n+\theta})_{n=0}^{N-1},  \label{eq:d2}
\end{gather}
for $v = (v^n) \in l^p(N;X)$.

Discrete maximal regularity is then introduced as follows (see \cite{Kem15}).

\begin{definition}[Discrete maximal regularity, DMR]
Let $p \in (1, \infty)$.
The operator $A$ has \emph{maximal $l^p$-regularity ($l^p$-DMR)} on $J_T$ if and
 only if, 
for every $g \in l^p(N_T; X)$, there exists a unique solution 
$u \in X^{N_T}$
of \eqref{eq:dcp} satisfying
\begin{equation}
\| u_\theta \|_{l^p_\tau(N_T; X)} + \| D_\tau u \|_{l^p_\tau(N_T; X)} + \| Au_\theta \|_{l^p_\tau(N_T; X)}
\le \Cdmr \| g_\theta \|_{l^p_\tau(N_T; X)},
\label{eq:dmr}
\end{equation}
uniformly with respect to $\tau$, where $\Cdmr > 0$ is independent of $g$.
We say that $A$ has \emph{discrete maximal regularity (DMR)} if $A$ has $l^p$-DMR for some $p \in (1, \infty)$.
\end{definition}

In \cite{Blu01}, Blunck considered the forward Euler method $(\theta = 0)$
and characterized DMR by developing a discrete
version of the operator-valued Fourier multiplier theorem. However, the dependence of $\tau$ on
DMR inequalities is not clear since only the case $\tau=1$ is studied.  
The backward Euler method $(\theta = 1)$ with an arbitrary time
increment $\tau$ is discussed in \cite{AshS94}. Ashyralyev and
Sobolevski\u{\i} provided no reasonable sufficient conditions for DMR. 
Consequently, those results cannot be applied straightforwardly to numerical analysis.
In contrast to those works, we gave sufficient conditions on $\tau$,
$\theta$, $A$ for DMR to hold in \cite{Kem15}. We recall the 
statement below (see 
Lemma \ref{thm:sufficient}).

Spatial discretization must be addressed next. 
We introduce the finite element
approximation $L_h$ of $\Delta$ in $H^1_0(\Omega)$ and prove that $L_h$
has CMR. Herein, $h$ denotes the size
parameter of a triangulation $\mathcal{T}_h$. As a matter of fact,
Geissert studied CMR for the finite element approximation of the second order
parabolic equations in the divergence form 
in \cite{Gei06,Gei07}. He considered a smooth convex domain
$\Omega$ and triangulations defined on a polyhedral approximation $\Omega_h$ of
$\Omega$. (For the Neumann boundary condition case, he considered the
exactly fitted triangulation.)
Therefore, combining those results with our Lemma \ref{thm:sufficient}, we are able to obtain DMR for the smooth domain case.  
In those works, the method of \cite{SchTW98} and \cite{ThoW00} for studying stability and
analyticity in $L^\infty$ norm is applied. He first derived some
estimates for the discrete Green function associated with the finite
element operator in parabolic annuli. Then he obtained some estimates
in the whole $\Omega$ by a dyadic decomposition technique. Consequently, the
proofs are quite intricate. Moreover, he applied several kernel
estimates for the Green function associated with a parabolic
equation. Therefore, the domain and coefficients should
be suitably smooth. 

In the present paper, we take a completely different approach. We
directly establish a discrete version of the method using pure imaginary
powers of operators developed by \cite{DorV87}. To this end, we consider
polyhedral domains and study  
the discrete Laplacian \emph{with mass-lumping} $A_h$ instead of the
standard discrete Laplacian since the positivity-preserving
property of the semigroup generated by $A_h$ (see Lemma
\ref{prop:disc-Laplacian1}) plays an important role in our analysis. Actually, the
standard discrete Laplacian has no such property (see \cite{ThoW08}).    
It must be borne in mind that the $L^q$ theory for the discrete Laplacian \emph{with
mass-lumping} is of great use in study of nonlinear problems, such as the finite element and finite volume approximation of the Keller-Segel system modelling chemotaxis (see
\cite{Sai07,Sai12,ZhoS16}). 

After having established CMR and DMR for $A_h$ (see Theorems 
\ref{cor:mr_disc-Laplacian}, \ref{cor:dmr_disc-Laplacian},
\ref{cor:dmr_disc-Lap_finite} and \ref{cor:dmr_disc-Lap_initial}), we
derive optimal order error estimates for the finite element
approximations combined with the implicit $\theta$
method to \eqref{eq:heat} (see Theorem 
\ref{thm:error_linear}). We address not only unconditionally stable
cases ($\theta\in [1/2,1]$), but also conditionally stable cases
($\theta\in [0,1/2)$). For the latter case, we give a useful sufficient
condition for the scheme to be stable. As a further application, we study the finite
element approximation for \eqref{eq:semilinear} and prove optimal order error
estimates (see Theorem \ref{thm:sl}). 
Since nonlinearity $f$ is assumed
to be only locally Lipschitz continuous, the solution $u$ might blow up in some sense.
Our error estimate is valid as long as $u$
exists in contrast to \cite{Gei07}. To achieve such an objective, we apply
the fractional powers of $-A_h$ and derive a sub-optimal error estimate in
the 
$L^\infty(\Omega\times (0,T))$ norm as an intermediate result (see
Theorem \ref{thm:infty}). Our proposed method is apparently new in the
literature.   
Some auxiliary results including discrete Gagliardo-Nirenberg and
Sobolev inequalities are also presented (see Lemmas \ref{lem:disc-GN} and
\ref{lem:disc-Sobolev}).

We learned about \cite{LeyV15,KavLL15} after completion of the present study.
The paper \cite{LeyV15} specifically examined the time-discrete version of $L^p$-$L^q$-maximal regularity for arbitrary $p, q \in [1,\infty]$, by discontinuous Galerkin time stepping (cf.~\cite{Tho06}) for parabolic problems. 
This result is valid for $p, q=1, \infty$.
However, they did not consider the R-boundedness of sets of operators, which plays an important role in the theory of maximal regularity developed by Weis \cite{Wei01}.
The main tools in \cite{LeyV15} were the smoothing properties of the continuous and discrete Laplace operators.
Consequently, their estimate invariably contained the logarithmic term, so that the optimal error estimate is never obtained.
It was established by a related work \cite{KavLL15} that arbitrary A-stable time-discretization preserves the time-discrete version of maximal $L^p$-regularity for abstract Cauchy problems and for $p \in (1,\infty)$.
These results were obtained via the theory of R-boundedness. 
It is therefore partially the same result of our previous work \cite{Kem15}.
An optimal error estimate was established only for semi-discrete backward Euler scheme for a semilinear parabolic problem.
In contrast to these works, we deal only with the finite difference scheme in time. 
However, our error estimate is optimal for fully discretized problems.

The plan of this paper is as follows.
In Sec.~\ref{section:main}, we introduce the notion of finite element
approximation and state main results (Theorems 
\ref{cor:mr_disc-Laplacian}--\ref{thm:infty}).
We summarize some preliminary results used in the proofs of Theorems in
Sec.~\ref{section:preliminary}. Some auxiliary lemmas related to MR, DMR and $A_h$
are described there. A useful sufficient condition for DMR to hold is also
described there (Lemma \ref{thm:sufficient}). In Sec.~\ref{section:proof1}, we prove Theorems 
\ref{cor:mr_disc-Laplacian}--\ref{cor:dmr_disc-Lap_initial} by a
discrete version of the method of \cite{DorV87} using pure imaginary
powers of operators. Auxiliary results, Lemmas
\ref{lem:positivity_disc-Lap}, \ref{thm:bip_disc-Laplacian} and
\ref{lem:inv_ineq}, themselves are of
interest. The proor of error estimate (Theorem \ref{thm:error_linear}) for
the linear equation \eqref{eq:heat} is described in Sec.~\ref{sec:lin}.   
The semilinear equation \eqref{eq:semilinear} is studied in
Sec.~\ref{section:app_semilinear}.
Therein, we also prove auxiliary results including discrete Gagliardo-Nirenberg, Sobolev
inequalities and provide useful results related to the fractional powers
of $A_h$. Combining those results, we prove the
final error estimate, Theorems \ref{thm:sl} and \ref{thm:infty}.

\section{Main results}
\label{section:main}

Throughout this paper, $\Omega$ is assumed to be 
a bounded polygonal or polyhedral domain in $\bR^d$, $d = 2,3$, with the
boundary $\partial\Omega$.
We follow the notation of \cite{AdaF03}. 
As an abbreviation, we write $L^{q}=L^{q}(\Omega)$,
$W^{s,q}=W^{s,q}(\Omega)$ and $H^{s}=W^{s,2}$ for 
$q\in[1,\infty]$ and $s>0$. We use $W^{1,q}_0=\{v\in W^{1,q}\mid
v|_{\partial\Omega}=0\}$ and $H^{1}_0=W^{1,2}_0$.
Generic positive constants which are independent of discretization
parameters, $h$ and $\tau$, are denoted as $C$.   
Their values might be different in each appearance.

Since the boundary $\partial\Omega$ is not smooth, we make the following {shape assumption on }$\Omega$.

\begin{assum}[Shape assumption on $\Omega$]\label{assum:regularity}
There exists $\mu > d$ satisfying
\begin{equation}
\| v \|_{W^{2,q}} \le C \| \Delta v \|_{L^q}, \quad \forall v \in W^{2,q}\cap W^{1,q}_0, 
\label{eq:regularity}
\end{equation}
for $q \in (1,\mu)$, where $C>0$ depends only on $\Omega$ and $q$.
\end{assum}
For example, if $\Omega$ is a convex polygonal domain in $\bR^2$, then
one can find $\mu > 2$ satisfying Assumption \ref{assum:regularity} ({see} \cite{Gri85}).

Let $\cT_h$ be a triangulation of $\Omega$ with the granularity 
parameter $h$ defined below.  
Hereinafter, a family $\cT$ of triangles or tetrahedra is a
triangulation of $\Omega$ if and only if 
\begin{enumerate}
\item each element of $\cT$ is an open triangle or tetrahedron in $\Omega$ and 
\begin{equation*}
\Omega = \operatorname{Int}\left( \bigcup_{K \in \cT}\overline{K} \right),
\end{equation*}
where $\operatorname{Int}(\cdot)$ is the interior part of a set,
\item any two elements of $\cT$ meet only in entire common faces (when $d=3$), sides or vertices.
\end{enumerate}
We use the following notations:
\begin{itemize}
\item $h = \max_{K\in\cT_h}h_K$; \quad $h_K = $ the diameter of a triangle or tetrahedron $K$;
\item $\overline{N}_h=$ the number of nodes of $\cT_h$;\quad $N_h=$ the number of interior nodes;
\item $\{P_j\}_{j=1}^{\overline{N}_h}=$ the nodes of $\cT_h$;\quad  
$\{P_j\}_{j=1}^{{N}_h}=$ the interior nodes. 
\end{itemize}

We assume the following.

\begin{assum}[Regularity of {$\{ \cT_h \}_h$}]\label{assum:reg}
There exists $\nu > 0$ such that
\begin{equation*}
h_K \le \nu \rho_K, \quad \forall K \in \cT_h, \quad \forall h>0,
\end{equation*}
where $\rho_K$ denotes the radius of the inscribed circle or sphere of $K$.
\end{assum}

Here we consider the $P_1$ finite element.
Let $V_h$ be the space of continuous functions on $\Omega$ 
which are affine in each element $K \in \cT_h$.
For every node $P_j$ $(j=1,\dots,\overline{N}_h)$, $\phi_j$ is the corresponding basis of $V_h$,
which satisfies $\phi_j(P_i) = \delta_{ij}$, where $\delta_{ij}$ is Kronecker's delta.
Namely, $V_h$ is the linear space spanned by $\{ \phi_j \}_{j=1}^{\overline{N}_h}$.
We also set 
\begin{equation*}
S_h = \{ v_h \in V_h \mid v_h|_{\partial\Omega} = 0 \} = \operatorname{span}\{ \phi_j \}_{j=1}^{N_h}.
\end{equation*}

Moreover, we presume that $\{\cT_h\}_h$ satisfies the following conditions if necessary.
\begin{description}
\item[(H1)] (Inverse assumption) There exists $\gamma > 0$ such that 
\begin{equation*}
h \le \gamma h_K , \quad \forall K \in \cT_h, \quad \forall h>0.
\end{equation*}
\item[(H2)] (Acuteness) For each $h>0$ and for each $i,j \in \{ 1,2,\dots, \overline{N}_h \}$ with $i \ne j$, 
\begin{equation}
\int_\Omega \nabla \phi_i \cdot \nabla \phi_j ~dx \le 0.
\label{eq:acute}
\end{equation}
\end{description}

\begin{remark}
In the two-dimensional case, let $\sigma \subset \Omega$ be an edge of the triangulation $\cT_h$
and $K$ and $L$ be the triangles of $\cT_h$ which meet in $\sigma$.
Assume that the nodes $P_i$ and $P_j$ be both endpoints of $\sigma$.
We denote the interior angle of $K$ opposite to the edge $\sigma$ by $\alpha_{i,j}^K$.
Then, the condition \eqref{eq:acute} is equivalent to the equation of
$\alpha_{i,j}^K + \alpha_{i,j}^L \le \pi$. See \cite[Corollary
 3.48]{KnaA03} for the detail.
\end{remark}

\begin{remark}[Discrete maximum principle]
\label{rem:dmp}
The condition (H2) is equivalent to the discrete maximum principle, i.e., the following conditions are equivalent.
\begin{enumerate}
\item The triangulation $\cT_h$ fulfills the acuteness condition.
\item Let $u_h \in V_h$ be the solution of the following problem for $f \in L^2$ and $g_h \in V_h$:
\begin{equation*}
\begin{cases}
(\nabla u_h, \nabla v_h)_{L^2} = (f, v_h)_{L^2}, \quad \forall v_h \in S_h, \\
u_h |_{\partial\Omega} = g_h .
\end{cases}
\end{equation*}
Then, $u_h \ge 0$ in $\Omega$ provided that $f \ge 0$ in $\Omega$ and $g_h \ge 0$ on $\partial\Omega$.
\end{enumerate}
See \cite[Theorem 3.49]{KnaA03} for details. 
\end{remark}

\begin{remark}\label{rem:sa}
When $q=2$, $A_h$ is a self-adjoint operator in $X_{h,2}$. Therefore, (H2) is not required in the following discussion. 
However, the condition (H1) is required for the inverse inequality, which implies $H^1$-stability of the $L^2$-projection (the equation \eqref{eq:Sobolev-stability_L2-projection}) and the discrete Gagliardo-Nirenberg type inequality (Lemma \ref{lem:disc-GN}).
Therefore, this condition is imposed for the consequences of \eqref{eq:Sobolev-stability_L2-projection} and Lemma \ref{lem:disc-GN}, for example, Theorems \ref{thm:error_linear}--\ref{thm:infty}, even if $q=2$.
\end{remark}

\begin{figure}[htb]
\centering
\includegraphics{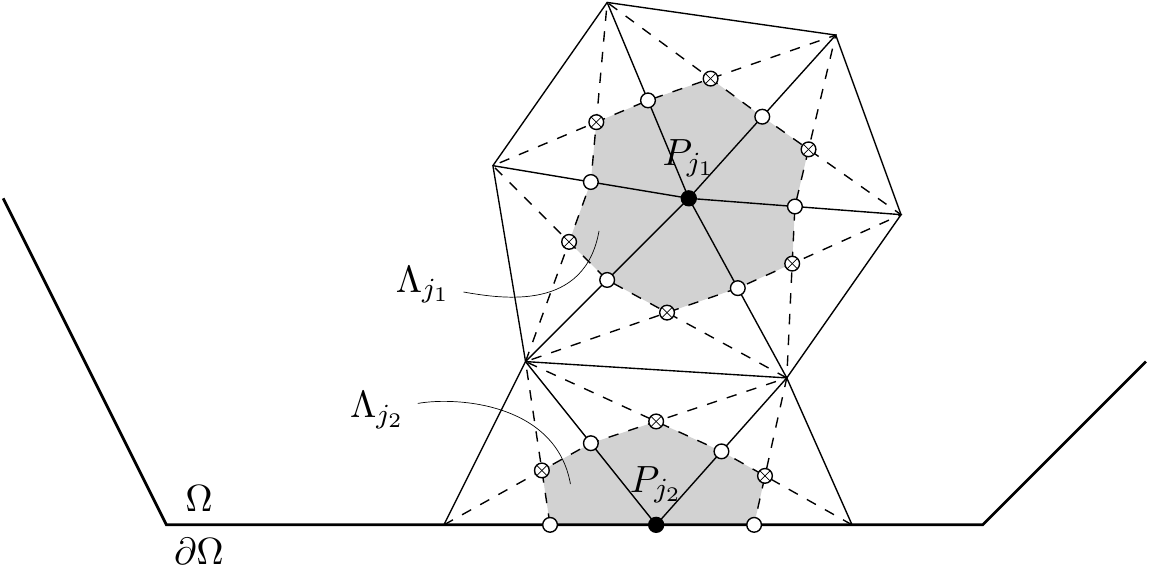}
\caption{Barycentric domains.
$P_{j_1}$: interior node, $P_{j_2}$: boundary node. 
\protect\includegraphics[page=1]{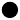}: node, 
\protect\includegraphics[page=2]{points.pdf}: midpoint of an edge, 
\protect\includegraphics[page=3]{points.pdf}: barycenter of a triangle.}
\label{fig:barycentric}
\end{figure}

We describe the method of mass-lumping.
For a node $P_j$, we designate the corresponding barycentric domain as
$\Lambda_j$; see Figure \ref{fig:barycentric} for illustration and see
\cite{FujSS01,Fuj73} for the definition. We denote the characteristic function of $\Lambda_j$ by $\chi_j$ for $j=1,\dots,N_h$.
Then, we set
\begin{equation*}
\overline{S}_h = \operatorname{span}\{ \chi_j \}_{j=1}^{N_h}
\end{equation*}
and define the lumping operator $M_h \colon S_h \to \overline{S}_h$ as
\begin{equation*}
M_h v_h = \sum_{j=1}^{N_h} v_h(P_j) \chi_j.
\end{equation*}
Moreover, we define $K_h = M_h^* M_h$, where $M_h^*$ is the adjoint operator of $M_h$
with respect to the $L^2$-inner product. As one might expect, $M_h$ is invertible and therefore $K_h$ is as well.
We define the mesh-dependent norms and inner product as
\begin{equation}
\| v_h \|_{h,q} = \| M_h v_h \|_{L^q}, \quad 
(u_h, v_h)_h = (M_h u_h, M_h v_h) ,\qquad
u_h, v_h \in S_h
\end{equation}
for $q \in [1,\infty]$. In fact, $\| \cdot \|_{h,q}$ is an equivalent norm to $\| \cdot \|_{L^q}$ in $S_h$
for each $q \in [1,\infty]$ (see Lemma \ref{lem:lumping}).

\medskip

At this stage, we introduce a discrete Laplacian as follows. Define the operator $A_h$ on $S_h$ as
\begin{equation}
(A_h u_h, v_h)_h = - (\nabla u_h, \nabla v_h), \quad \forall v_h \in S_h, \label{eq:dl}
\end{equation}
for $u_h \in S_h$. We designate $A_h$ the discrete Laplacian \emph{with}
mass-lumping. From the Poincar\'e inequality, $A_h$ is injective so
that it is invertible due to $\dim S_h < \infty$.

\medskip

We are now in a position to state the main results of this study. In the
theorems below, we always presume that
Assumptions \ref{assum:regularity} and \ref{assum:reg} are satisfied, unless otherwise stated explicitly.
The first one is about CMR for $A_h$.

\begin{thm}[CMR for $A_h$]\label{cor:mr_disc-Laplacian}
Let $T \in (0,\infty]$, $p \in (1, \infty)$ and $q \in (1, \mu)$.
Assume that (H1) and (H2) are satisfied when $q \ne 2$. 
Then, $A_h$ has $L^p$-CMR on $J_T$ in $X_{h,q}$ uniformly for $h > 0$.
That is, there exists $C>0$ independent of $h>0$ satisfying
\begin{equation}
\| u_h \|_{L^p(J_T; X_{h,q})} + \| u_h' \|_{L^p(J_T; X_{h,q})} + \| A_h u_h \|_{L^p(J_T; X_{h,q})}
\le C \| g_h \|_{L^p(J_T; X_{h,q})},
\label{eq:mr_disc-Lap}
\end{equation}
where $g_h \in L^p(J_T; X_{h,q})$ and $u_h$ is the solution of
\begin{equation}\label{eq:00}
\begin{cases}
u_h'(t) = A_h u_h(t) + g_h(t), & t \in J_T, \\
u_h(0) = 0.
\end{cases}
\end{equation}
\end{thm}

\begin{remark}\label{rem:fd}
Since \eqref{eq:00} is a system of (inhomegeneous) linear ordinary
 differential equations, the unique existence of a solution follows
 immediately.   
\end{remark}

Next, we state results about DMR for $A_h$. To state them, we set
\begin{align}
 \theta_q&=\arccos |1-q/2|,\label{eq:tq}\\
\kappa_h &= \min_{K \in \cT_h} \kappa_K, \label{eq:kappa}
 \end{align}
where $\kappa_K $ denotes the minimum length of perpendiculars of $K$.

\begin{thm}[DMR for $A_h$ in $J_\infty$]\label{cor:dmr_disc-Laplacian}
Let $p \in (1, \infty)$, $q \in (1, \mu)$ and $\theta \in [0,1]$.
Assume that (H1) and (H2) are satisfied when $q \ne 2$. 
 We choose $\varepsilon$ and $\tau$ sufficiently small to satisfy
\begin{equation}
\frac{\tau}{\kappa_h^2} \le \frac{2 \sin \theta_q - \varepsilon}{(1 -2 \theta)(d+1)^2},
\label{eq:sufficient-NR}
\end{equation}
when $\theta \in [0, 1/2)$.
Then, $A_h$ has $l^p$-DMR on  $J_\infty$ in $X_{h,q}$ uniformly for $h > 0$.
That is, there exists $C>0$ independent of $h$ and $\tau$ satisfying
\begin{equation*}
\| u_{h,\theta} \|_{l^p_\tau(\bN; X_{h,q})} + \| D_\tau u_h \|_{l^p_\tau(\bN; X_{h,q})} + \| A_h u_{h,\theta} \|_{l^p_\tau(\bN; X_{h,q})}
\le C \| g_{h,\theta} \|_{l^p_\tau(\bN; X_{h,q})},
\end{equation*}
where $g_h = (g_h^n)_n \in l^p(\bN; X_{h,q})$ and $u_h = (u_h^n)_n$ is the solution of
\begin{equation*}
\begin{cases}
(D_\tau u_h)^n = A_h u_h^{n + \theta} + g_h^{n+\theta}, & n \in \bN, \\
u_h^0 = 0.
\end{cases}
\end{equation*}
\end{thm}

\begin{thm}[DMR for $A_h$ in $J_T$]
\label{cor:dmr_disc-Lap_finite}
Let $p \in (1, \infty)$, $q \in (1, \mu)$ and $\theta \in [0,1]$.
Assume that (H1) and (H2) are satisfied when $q \ne 2$.
Choose $\varepsilon$ and $\tau$ sufficiently small to satisfy \eqref{eq:sufficient-NR}, when $\theta \in [0, 1/2)$.
Then, for every $T>0$ and for every $g_h \in l^p(N_T-1;X_{h,q})$,
there exists a unique solution $u_h \in l^p(N_T;X_{h,q})$ of 
\begin{equation*}
\begin{cases}
(D_\tau u_h)^n = A_h u_h^{n + \theta} + g_h^{n}, & n =0,1,\dots, N_T-1, \\
u_h^0 = 0,
\end{cases}
\end{equation*}
and it satisfies
\begin{equation}
\| u_{h,\theta} \|_{l^p_\tau(N_T; X_{h,q})} + \| D_\tau u_h \|_{l^p_\tau(N_T; X_{h,q})} + \| Au_{h,\theta} \|_{l^p_\tau(N_T; X_{h,q})}
\le C \| g_h \|_{l^p_\tau(N_T; X_{h,q})},
\label{eq:dmr-var_heat}
\end{equation}
where $C>0$ is independent of $g$, $T$, $h$, and $\tau$.\end{thm}

\begin{thm}[DMR for non-zero initial value]
\label{cor:dmr_disc-Lap_initial}
Let $p \in (1, \infty)$ and $q \in (1, \mu)$.
Assume that (H1) and (H2) are satisfied when $q \ne 2$.
Then, for every $T>0$, $g_h \in l^p(N_T;X_{h,q})$, and $u_{0,h} \in (X_{h,q}, D(A_h))_{1-1/p,p}$,
there exists a unique solution $u_h \in l^p(N_T;X_{h,q})$ of 
\begin{equation*}
\begin{cases}
(D_\tau u_h)^n = A_h u_h^{n+1} + g_h^{n+1}, & n =0,1,\dots, N_T-1, \\
u_h^0 = u_{0,h},
\end{cases}
\end{equation*}
which satisfies
\begin{multline}
\| u_{h,1} \|_{l^p_\tau(N_T; X_{h,q})} + \| D_\tau u_h \|_{l^p_\tau(N_T; X_{h,q})} + \| Au_{h,1} \|_{l^p_\tau(N_T; X_{h,q})}\\
\le C \left[  \| g_{h,1} \|_{l^p_\tau(N_T; X_{h,q})} + \|u_{0,h}\|_{1-1/p,p}
\right],
\label{eq:dmr-initial_heat}
\end{multline}
where $C>0$ is independent of $g$, $u_{0,h}$, $T$, $h$, and $\tau$.
\end{thm}

Therein, $(X_{h,q}, D(A_h))_{1-1/p,p}$ and $\| \cdot \|_{1-1/p,p}$ respectively denote
the real interpolation space and its norm. 
(see Subsection \ref{subsection:embedding} for related details.)

\medskip

Those theorems are applicable for error analysis of
the fully discretized finite element approximation for heat equations.
First, we consider a linear heat equation \eqref{eq:heat} 
for $T \in (0,\infty)$, $g \in L^p(J_T; L^q)$ and $u_0 \in L^q$.
We further assume $g \in C^0(\overline{J}_T; L^q))$.
We consider the following approximate problem to find $u_h = (u_h^n)_n \in l^p(N_T; S_h)$ satisfying
\begin{equation}\label{eq:disc-heat0}
 \begin{cases}
((D_\tau u_h)^n, v_h)_h +(\nabla u_h^{n+\theta}, \nabla v_h)_{L^2} 
  = (G^{n+\theta}, v_h)_{L^2}, & \forall v_h \in S_h,\\
  & n =0,1\dots,N_T-1,\\
(u^{0}_h, v_h)_{L^2}=(u_0, v_h)_{L^2} , & \forall v_h \in S_h
\end{cases}
\end{equation}
where $\tau \in (0,1)$, $\theta \in [0,1]$, $t_n = n\tau$, and $G^n =
g(\cdot, t_n)$. An alternative scheme is obtained with replacement
$(G^{n+\theta}, v_h)_{L^2}$ by $(G^{n+\theta}, v_h)_h$. However, the
resulting scheme has a shortcoming reported in Appendix \ref{sec:scheme}.

Let $P_h$ be the $L^2$-projection onto $S_h$ defined as 
\begin{equation}
(P_h v, v_h)_{L^2} = (v, v_h)_{L^2}, \quad \forall v_h \in S_h  \label{eq:Ph}
\end{equation}
for $v\in L^1$. 

Then, \eqref{eq:disc-heat0} is equivalently written as 
\begin{equation}
\begin{cases}
(D_\tau u_h )^n = A_h u_h^{n+\theta} + K_h^{-1} P_h G^{n+\theta}, & n =0,1,\dots, N_T - 1, \\
u_h^0 = P_h u_0.
\end{cases} \label{eq:disc-heat}
\end{equation}
Since $A_h$ is invertible, there exists a unique
solution of \eqref{eq:disc-heat}.
We introduce
\begin{equation}
\label{eq:jt}
 j_\theta =
\begin{cases}
 2, & \theta = 1/2,\\
 1, & \mbox{otherwise}
\end{cases}
\end{equation}
and $\mu_d = \max\{ \mu', d/2 \}$.
Since $\mu' < d' \le 2 \le d < \mu$, it might be apparent that
\begin{equation*}
\mu_d = \begin{cases}
\mu' , & d = 2, \\ d/2 = 3/2 , & d = 3.
\end{cases}
\end{equation*}

\begin{thm}[Error estimate for linear equation]
\label{thm:error_linear}
Let $p \in (1, \infty)$ and $q \in (\mu_d, \mu)$.
Let $u_h = (u_h^n)_n \in l^p(N_T;S_h)$ be the solution of \eqref{eq:disc-heat} and $u$ be that of \eqref{eq:heat}.
 Assume $u\in W^{1,p}(J_T;W^{2,q})\cap W^{2,p}(J_T;W^{1,q})\cap
 W^{1+j_\theta,p}(J_T;L^{q})$ and set $U^n = u(\cdot, t_n)$. 
Assume that (H1) and (H2) are satisfied. 
Moreover, we choose $\varepsilon$ and $\tau$ sufficiently small to satisfy \eqref{eq:sufficient-NR}, when $\theta \in [0, 1/2)$.
Then, there exists a positive constant $C$ such that
\begin{equation}
\left( \sum_{n=0}^{N_T-1} \| u_h^{n+\theta} - U^{n+\theta} \|_{L^q}^p \tau \right)^{1/p} \le C( h^2 + \tau^{j_\theta}).
\label{eq:error_estimate}
\end{equation}
The constant $C$ is taken as
\[
 C=C'\left(
 \|u\|_{W^{1,p}(J_T;W^{2,q})}+
 \|\partial_tu\|_{W^{1,p}(J_T;W^{1,q})}+
 \|u\|_{W^{1+j_\theta,p}(J_T;L^{q})}
  \right),
\]
where $C'$ depends only on $\Omega$, $p$, $q$, and $\theta$, but is
 independent of $h$ and $\tau$. 
\end{thm}

For $q \in (1, \infty)$, let $A_q$ be the realization of the Dirichlet Laplacian:
\begin{equation}
\label{eq:delq}
D(A_q) = W^{2,q} \cap W^{1,q}_0, \quad A_q u = \Delta u.
\end{equation}
We are assuming Assumption \ref{assum:regularity}. 
We consider a semilinear heat equation \eqref{eq:semilinear} under the
following basic assumptions: 
\begin{gather}
u_0 \in (L^q, D(A_q))_{1-1/p, p}, \label{eq:a1}\\
 f \colon \bC \to \bC\text{ is locally Lipschitz continuous
  with }f(0) = 0
. \label{eq:a2}
\end{gather}
Herein, $(L^q, D(A_q))_{1-1/p, p}$ denotes the real
interpolation space \cite{Ama95,Lun09,Tri95}.
Restriction $f(0) = 0$ is set for simplicity.
It is noteworthy that the solution $u$ of \eqref{eq:semilinear} might blow-up:
let $T_\infty \in (0, \infty]$ be the life span of $u$ (the maximal
existence time of $u$).

To avoid unnecessary difficulties, we restrict our consideration to a semi-implicit scheme for \eqref{eq:semilinear} given as
\begin{equation}
\begin{cases}
(D_\tau u_h)^n = A_h u_h^{n+1} + K_h^{-1} P_h f(u_h^n), & n=0,1,\ldots,N_T-1,\\
u_h^0 = P_h u_0,
\end{cases}
\label{eq:approx-semilinear}
\end{equation}
or, equivalently,
\begin{equation*}
\begin{cases}
((D_\tau u_h)^n, v_h)_h +(\nabla u_h^{n+1}, \nabla v_h)_{L^2} 
 = (f(u_h^{n}), v_h)_{L^2},  & \forall v_h \in S_h,\\
 & n=0,1,\ldots,N_T-1,\\
(u_h^0,v_h)_{L^2} = (u_0,v_h)_{L^2}, & \forall v_h \in S_h.
\end{cases}
\end{equation*}

Since $A_h$ is invertible, there exists a unique
solution of \eqref{eq:approx-semilinear}. Our final theorem is the
following error estimate for semilinear equation. Our error estimate
remains valid as long as the solution of \eqref{eq:semilinear} exists
and requires no size condition on $u_0$.

 \begin{thm}[Error estimate for semilinear equation]
 \label{thm:sl}
Let $p\in (1,\infty)$, $q \in (\mu_d, \mu)$ and $p > 2q/(2q-d)$.
Assume that (H1) and (H2) are satisfied. 
Presuming that
  \eqref{eq:semilinear} admits a sufficiently smooth solution $u$ under
  the conditions \eqref{eq:a1} and \eqref{eq:a2}, 
  then, for every $T \in (0, T_\infty)$ and
the solution $u_h = (u_h^n)_{n=0}^{N_T}$ of
  \eqref{eq:approx-semilinear},
  we have  
\begin{equation*}
\left( \sum_{n=1}^{N_T} \| u_h^n - U^n \|_{L^q}^p \tau \right)^{1/p}
\le C(h^2 + \tau),
\end{equation*}
where $U^n=u(\cdot,n\tau)$. 
 \end{thm}

In the proof of Theorem \ref{thm:sl} (Sec. \ref{section:app_semilinear}), the following sub-optimal
error estimate, which is worth stating separately, will be used. 
 
\begin{thm}[$L^\infty$ error estimate for semilinear equation]
 \label{thm:infty}
 Under the same assumptions of Theorem \ref{thm:sl},
 for every $\alpha \in (0, \alpha_{p,q,d})$ and $T \in (0, T_\infty)$, 
the following error estimate holds:
\begin{equation*}
\max_{0 \le n \le N_T} \| {u}_h^n - U^n \|_{L^\infty}
\le C(h^{2\alpha}+\tau),
\end{equation*}
where $\alpha_{p,q,d} = 1 - 1/p - d/(2q)$ and $U^n=u(\cdot,n\tau)$.
\end{thm}

\section{Preliminaries}
\label{section:preliminary}

As explained in this section, we collect some preliminary results used for this study.

\subsection{Continuous maximal regularity}
\label{sec:cmr}

{The definition of CMR in Definition \ref{def:mr1}} is the classical one.
The weaker one is introduced in \cite[Definition 4.1]{Wei01}, 
which requires the inequality
\begin{equation}
\| u' \|_{L^p(J_T; X)} + \| Au \|_{L^p(J_T; X)}
\le C \| f \|_{L^p(J_T; X)}
\label{eq:mr-weak}
\end{equation}
instead of \eqref{eq:mr}.
Also, CMR in this sense is characterized by operator-theoretical properties 
(\cite[Theorem 4.2]{Wei01}).
However, two inequalities \eqref{eq:mr} and \eqref{eq:mr-weak} are
equivalent if $0 \in \rho(A)$, where $\rho(A)$ denotes the resolvent set of $A$.
Since the condition $0 \in \rho(A)$ is satisfied in our application,
we ignore the differences between these definitions.

Conditions necessary for CMR to hold have been studied by many researchers (see e.g.~\cite{Dor93,Wei01}).
Among them, we review some sufficient conditions for CMR, 
which will be used for this study.
For the detail, see \cite{Dor93} and references therein.

\begin{lemma}\label{lem:LpMR}
Let $T\in (0,\infty]$, $X$ be a Banach space and 
$A$ be a densely defined and closed operator on $X$.
Assume that $A$ has $L^{p_0}$-CMR on $J_T$ for some $p_0 \in (1, \infty)$.
Then, $A$ has $L^p$-CMR on $J_T$, for any $p \in (1, \infty)$.
\end{lemma}

\begin{lemma}\label{lem:mr-interval}
Let $p \in (1,\infty)$, $X$ be a Banach space and let 
$A$ be a densely defined and closed operator on $X$.
Assume that $A$ has $L^p$-CMR on $J_\infty$.
Then, $A$ has $L^p$-CMR on $J_T$, for any $T>0$.
\end{lemma}

The next lemma is the celebrated result of Dore and Venni \cite[Theorem 3.2]{DorV87} (see also \cite[Section III.4]{Ama05}).

\begin{lemma}\label{lem:mr-imaginary_power}
Let $p \in (1,\infty)$, $X$ be a UMD space, and let 
$A$ be a densely defined and closed operator on $X$. Assume that $-A \in \cP(X;K) \cap \bip(X;M,\theta)$ for some $K>0$, $M \ge 1$, and $\theta \in [0, \pi/2)$.
Then, $A$ has $L^p$-CMR on $J_T$, for any $T>0$ and $T=\infty$.
Moreover, the constant $\Cmr > 0$ depends only on $X$, $K$, $M$, $\theta$, and $T$. 
\end{lemma}

Herein, the sets $\cP(X; K)$ and $\bip(X; M, \theta)$ are defined as
\begin{align*}
\cP(X; K) &= \left\{ A \in \cC(X) \mid 
 \rho(A) \subset (-\infty, 0] \text{ and } \right. \\
 & \mbox{ }\qquad \qquad \qquad \qquad \left. \| (1+\lambda) R(\lambda; A) \|_{\cL(X)} \le K, \
\forall \lambda \ge 0 \right\}, \\
\bip(X; M, \theta) &= \{ A \in \cP(X) \mid 
A^{it} \in \cL(X) \text{ and } 
\| A^{it} \|_{\cL(X)} \le M e^{\theta |t|} ,\ 
\forall t \in \bR  \},
\end{align*}
for $K > 0$, $M \ge 1$, and $\theta \ge 0$,
where $\cC(X)$ is the set of all closed linear operators on $X$ with dense domains, $\cP(X) = \bigcup_{K>0}\cP(X;K)$. 
The imaginary power $A^{it}$ is defined by $H^\infty$-functional calculus 
(see Appendix \ref{section:functional_calculus}).

The dependence of the constant $\Cmr$ on the Banach space $X$ 
derives from the boundedness of
imaginary powers of the time-differential operator on $L^p(J_T; X)$.
See \cite[Lemma III.4.10.5]{Ama05} for $T < \infty$ and 
\cite[Corollary 8.5.3]{Haa06} for $T = \infty$. 
Chasing the constants appearing in the proofs, we can obtain the
following property (see \cite{Kem15}).

\begin{lemma}\label{cor:mr-imaginary_power}
Let $p \in (1,\infty)$, $X$ be a UMD space, $X_0 \subset X$ be a closed subspace, and 
$A$ be a densely defined and closed operator on $X_0$.
Assume that $-A \in \cP(X_0;K) \cap \bip(X_0;M,\theta)$ for some $K>0$, $M \ge 1$, and $\theta \in [0, \pi/2)$. 
Then
$A$ has $L^p$-CMR on $J_T$, for any $T>0$ and $T=\infty$.
Moreover, the constant $\Cmr > 0$ depends only on $X$, $K$, $M$, $\theta$, and $T$, but is independent of $X_0$.
\end{lemma}

In the definition of CMR \eqref{eq:acp}, 
we consider only the zero initial value.
However, in general cases, particularly in the nonlinear cases, the choice of initial values is extremely important.
Therefore, we now consider the following Cauchy problem:
\begin{equation}
\begin{cases}
u'(t) = Au(t) + g(t), & t \in J_T, \\
u(0) = u_0,
\end{cases}
\label{eq:acp2}
\end{equation}
for $u_0 \in X$.

\begin{lemma}
Let $p \in (1,\infty)$, $T \in (0,\infty] $, $X$ be a Banach space and $A$ be a densely defined and closed operator.
Assume that $A$ has $L^p$-CMR on $J_T$.
Then, for each $g \in L^p(J_T; X)$ and for each $u_0 \in (X, D(A))_{1-1/p,p}$, there exists a unique solution $u \in W^{1,p}(J_T; X) \cap L^p(J_T; D(A))$
of \eqref{eq:acp2} satisfying
\begin{multline}
\| u \|_{L^p(J_T; X)} + \| u' \|_{L^p(J_T; X)} + \| Au \|_{L^p(J_T; X)} \\
\le \Cmr \left( \| g \|_{L^p(J_T; X)} + \|u_0 \|_{1-1/p,p} \right),
\label{eq:mr2}
\end{multline}
where $\Cmr > 0$ is independent of $g$ and $u_0$.
\end{lemma}
Herein, the norm $\| \cdot \|_{1-1/p,p}$ is the norm of the real interpolation space $(X, D(A))_{1-1/p,p}$.

\subsection{Discrete maximal regularity}
\label{sec:dmr}

As in the CMR case, the weaker definition can be considered, which does not require that $0 \in \rho(A)$.
Indeed, the weaker one is used in \cite{Blu01,Kem15}.
However, for the same reason as that presented in the previous subsection, we do not distinguish these two definitions.

We investigated a sufficient condition for DMR on $J_\infty$, in the UMD case in \cite{Kem15}. More precisely, we proved
the following result. 

\begin{lemma}\label{thm:sufficient}
Let $p \in (1,\infty)$, $\theta \in [0,1]$, $X$ be a UMD space, $X_0 \subset X$ be a closed subspace, and $A$ be a bounded operator on $X_0$.
Assume that $A$ has $L^p$-CMR on $J_\infty$ with the constant $\Cmr$.
Furthermore, we suppose that the following conditions (condition $(\text{NR})_{\delta, \varepsilon}$) are satisfied 
when $\theta \in [0, 1/2)$:
\begin{description}
\item[(NR1)] There exists $\delta \in (0,\pi/2)$ such that $S(A) \subset \bC \setminus \Sigma_{\delta + \pi/2}$.
\item[(NR2)] There exists $\varepsilon > 0$ such that $(1-2\theta)\tau r(A) + \varepsilon \le 2 \sin \delta$.
\end{description}
Then, $A$ has $l^p$-DMR on $J_\infty$.
Moreover, the constant $\Cdmr$ depends only on $p$, $\theta$, $\delta$, $\varepsilon$, $X$, and $\Cmr$, but is independent of $X_0$.
\end{lemma}

Herein, for $\omega\in (0,\pi)$, the set $\Sigma_\omega$ denotes the sector
\begin{equation}
 \Sigma_\omega=\{z\in\mathbb{C}\backslash\{0\}\mid~ |\arg z|<\omega\}.
 \label{eq:sec}
\end{equation}
The set $S(A) \subset \bC$ is the numerical range of $A$ defined as
\begin{equation}
S(A) = \left\{ 
\langle x^*, Ax \rangle \lmid 
\begin{array}{l}
x \in D(X),\ \|x\| = 1, \\
x^* \in X^*,\ \|x^*\| = 1,\ \langle x^*, x \rangle = 1.
\end{array}
\right\},
\label{eq:num_ran1}
\end{equation}
where $\langle \cdot, \cdot \rangle$ is the duality paring (\cite{FujSS01,Paz83}). We set
\[
r(A) = \max_{z \in S(A)} |z|.
\]

Actually, DMR on finite intervals is obtainable from the infinite-interval case.
The following lemma corresponds to Lemma \ref{lem:mr-interval}.
Although the inequality \eqref{eq:dmr-var} below is slightly different
from \eqref{eq:dmr}, it does not affect error analysis.

\begin{lemma}\label{lem:dmr_finite}
Let $p \in (1,\infty)$, $\theta \in [0,1]$, $X$ be a Banach space, and $A$ be a bounded operator on $X$.
Assume that $A$ has $l^p$-DMR on $J_\infty$ with $\Cdmr = C_0$.
 Then, for every $T>0$ and for every $g \in l^p(N_T-1;X)$,
there exists a unique solution $u \in l^p(N_T;X)$ of the equation
\begin{equation}
\begin{dcases}
\frac{u^{n+1} - u^n}{\tau} = Au^{n+\theta} + g^{n}, & n = 0,1,\dots, N_T -1, \\
u^0 = 0,
\end{dcases}
\label{eq:dcp3}
\end{equation}
and it satisfies
\begin{equation}
\| u_\theta \|_{l^p_\tau(N_T; X)} + \| D_\tau u \|_{l^p_\tau(N_T; X)} + \| Au_\theta \|_{l^p_\tau(N_T; X)}
\le C_0 \| g \|_{l^p_\tau(N_T; X)}.
\label{eq:dmr-var}
\end{equation}
\end{lemma} 

\begin{proof}
Fix $T>0$, $\tau>0$, and $g \in l^p(N_T-1; X)$ arbitrarily.
Define $\tilde{g} \in l^p(\bN; X)$ as
\begin{equation*}
\tilde{g}^n = \begin{cases}
g^n, & n = 0,1, \dots, N_T-1, \\
0, & n \ge N_T,
\end{cases}
\end{equation*}
and consider the Cauchy problem
\begin{equation*}
\begin{dcases}
\frac{\tilde{u}^{n+1} - \tilde{u}^n}{\tau} = A\tilde{u}^{n+\theta} + \tilde{g}^{n}, & n = 0,1,\dots, \\
\tilde{u}^0 = 0.
\end{dcases}
\end{equation*}
Since $A$ has $l^p$-DMR on $J_\infty$, we can find the corresponding solution $\tilde{u} = (\tilde{u}^n)_n \in X^\bN$ satisfying
\begin{equation*}
\| \tilde{u}_\theta \|_{l^p_\tau(\bN; X)} + \| D_\tau \tilde{u} \|_{l^p_\tau(\bN; X)} + \| A\tilde{u}_\theta \|_{l^p_\tau(\bN; X)}
\le \Cdmr \| \tilde{g} \|_{l^p_\tau(\bN; X)}.
\end{equation*}
Then, $u := (\tilde{u}^n)_{n=0}^{N_T} \in l^p(N_T;X)$ is a solution of \eqref{eq:dcp3}, and fulfills
\begin{equation*}
\| u_\theta \|_{l^p_\tau(N_T; X)} + \| D_\tau u \|_{l^p_\tau(N_T; X)} + \| Au_\theta \|_{l^p_\tau(N_T; X)}
\le \Cdmr \| \tilde{g} \|_{l^p_\tau(\bN; X)},
\end{equation*}
which implies \eqref{eq:dmr-var}.
The uniqueness of the solution might be readily apparent.
\qed\end{proof}

An a priori estimate with non-zero initial value is obtained only in the
case where $\theta = 1$. 
See \cite{AshS94} for $T<\infty$ and \cite{Kem15} for $T=\infty$.

\begin{lemma}\label{lem:dmr_initial}
Let $p \in (1,\infty)$, $\theta \in [0,1]$, $T \in (0,\infty]$, $X$ be a UMD space, $X_0 \subset X$ be a closed subspace, and $A$ be a bounded operator on $X_0$.
Assume that $A$ has $l^p$-DMR on $J_T$.
Then, for each $g \in l^p(N_T; X_0)$ and for each $u_0 \in (X_0, D(A))_{1-1/p,p}$, there exists a unique solution $u \in l^p(N_T; X_0)$
of the equation
\begin{equation*}
\begin{dcases}
\frac{u^{n+1} - u^n}{\tau} = Au^{n+1} + g^{n+1}, & n = 0,1,\dots, N_T -1, \\
u^0 = u_0,
\end{dcases}
\end{equation*} 
which satisfies
\begin{multline*}
 \| u_1 \|_{l^p_\tau(N_T; X_0)} + \| D_\tau u \|_{l^p_\tau(N_T; X_0)} +
 \| Au_1 \|_{l^p_\tau(N_T; X_0)} \\
\le \Cdmr \left( \| g_1 \|_{l^p_\tau(N_T; X_0)} + \|u_0 \|_{1-1/p,p} \right),
\end{multline*}
where $\Cdmr > 0$ is independent of $g$, $u_0$, and $X_0$.
\end{lemma}

\subsection{Operator-theoretical properties of $A_h$}
\label{sec:ah}

A semigroup $T(t)$ on a Lebesgue space $X = L^q(\Omega,\mu)$ ($q \in [1,\infty]$) is said to be positivity-preserving if  
\begin{equation*}
u \ge 0 \text{ $\mu$-a.e.\ in } \Omega \implies T(t) u \ge 0 \text{ $\mu$-a.e.\ in } \Omega 
\end{equation*}
for each $t>0$ and $u \in X$.
In the proofs of the following two lemmas, the discrete maximum
principle (Remark \ref{rem:dmp}) plays a crucially important role.

\begin{lemma}[\mb{\cite[Theorem 15.5]{Tho06}}]
\label{prop:disc-Laplacian1}
Let $q \in [1, \infty]$.
Assume that the family of triangulations $\{\cT_h\}$ satisfies the acuteness condition (H2).
Then, the semigroup $e^{t{A_h}}$ generated by $A_h$ is positivity-preserving in $X_{h,q}$.
\end{lemma}

\begin{lemma}[\mb{\cite[Theorem 4.1]{CroT01}}]
\label{prop:disc-Laplacian2}
Let $q \in [1, \infty]$.
Assume that the family of triangulations $\{\cT_h\}$ satisfies the acuteness condition (H2).
Then, $A_h$ generates an analytic and contraction semigroup on $X_{h,q}$.
Moreover, if $q \in (1, \infty)$, then $A_h$ satisfies the condition
 (NR1) with the angle $\theta_q$ defined as \eqref{eq:tq}. 
\end{lemma}

We introduce several mesh-depending operators on $S_h$.
The $L^2$ projection $P_h$ is defined as \eqref{eq:Ph}.
Let $R_h$ be the Ritz projection of $W^{1,1}\to S_h$ defined as 
\begin{equation*}
(\nabla R_h u, \nabla v_h)_{L^2} = (\nabla u, \nabla v_h)_{L^2}, \quad \forall v_h \in S_h
\end{equation*}
for $u \in W^{1,1}$.
These operators have the following well-known properties. 
See \cite{DouDW75,CroT87,BreS08} for the proofs.

\begin{lemma}\label{lem:projections}
Assume that $\{\cT_h\}_h$ satisfies (H1).
Then, there exists $C>0$ depending only on $\Omega$ and $q$ such that
\begin{align}
\| P_h v \|_{L^q} &\le C \| v \|_{L^q}, \quad \forall v \in L^q, 
\quad \forall q \in [1,\infty], \\
\| P_h v \|_{W^{1,q}} &\le C \| v \|_{W^{1,q}}, \quad \forall v \in W^{1,q}, 
\quad \forall q \in [1,\infty], \label{eq:Sobolev-stability_L2-projection} \\
\| R_h v \|_{W^{1,q}} &\le C \| v \|_{W^{1,q}}, \quad \forall v \in W^{1,q}, 
\quad \forall q \in (1,\infty], \\
\| v - P_h v \|_{L^q} &\le C h^2 \| v \|_{W^{2,q}}, \quad \forall v \in W^{2,q}, 
\quad \forall q \in (d/2,\infty], \\
\| v - R_h v \|_{L^q} &\le C h^2 \| v \|_{W^{2,q}}, \quad \forall v \in D(A_q), 
\quad \forall q \in (\mu', \infty), 
\end{align}
where $\mu'$ is the H\"{o}lder conjugate of $\mu$. 
When $q \ne 2$, (H1) is not required for all inequalities above except for \eqref{eq:Sobolev-stability_L2-projection}.
\end{lemma}

Mass-lumping operator $M_h$ and $K_h$ have the following properties.
For the proof, see \cite{FujSS01}.

\begin{lemma}\label{lem:lumping}
Let $q \in [1,\infty]$. Then, there exists $C>0$ depending only on $q$ and $\Omega$ such that
\begin{equation*}
C^{-1} \| v_h \|_{L^q} \le \| M_h v_h \|_{L^q} \le C \| v_h \|_{L^q}, 
\quad v_h \in S_h, \quad q \in [1,\infty].
\end{equation*}
Moreover, if $\{ \cT_h \}_h$ satisfies (H1) when $q \ne 2$,
then there exists $C>0$ depending only on $q$ and $\Omega$ such that
\begin{equation*}
C^{-1} \| v_h \|_{L^q} \le \| K_h v_h \|_{L^q} \le C \| v_h \|_{L^q}, 
\quad v_h \in S_h, \quad q \in [1,\infty]. 
\end{equation*}
\end{lemma}

We use the standard discrete Laplacian $L_h$ defined as
\[
 (L_h u_h, v_h) = - (\nabla u_h, \nabla v_h), \quad \forall v_h \in S_h,
\]
for $u_h \in S_h$.
We designate $L_h$ the discrete Laplacian \emph{without} mass-lumping.
From the Poincar\'e inequality, $L_h$ is injective. 
Consequently, it is invertible due to $\dim S_h < \infty$.
Then, by the definitions given above, it is apparent that
\begin{equation}
L_h = K_h A_h, \quad R_h = L_h^{-1} P_h A. \label{eq:disc-Lap}
\end{equation}
From these relations, the following estimate is obtained.
\begin{lemma}\label{lem:inv_disc-Lap}
Assume that $\{ \cT_h \}_h$ satisfies (H1) when $q \ne 2$.
Then, for $q \in (1, \mu)$, there exists $C>0$ satisfying
\begin{equation*}
\| v_h \|_{h,q} \le C \| A_h v_h \|_{h,q}, \quad \forall v_h \in S_h,
\end{equation*}
where $C$ depends only on $\Omega$ and $q$.
\end{lemma}

\begin{proof}
By \eqref{eq:disc-Lap} and Lemma \ref{lem:lumping}, it suffices to show that
\begin{equation*}
\| v_h \|_{L^q} \le C \| L_h v_h \|_{L^q}
\end{equation*}
for all $v_h \in S_h$.
Fix $v_h \in S_h$ arbitrarily and set $f_h = L_h v_h$ and $v = A^{-1}f_h \in D(A)$.
Then, noting that $P_h f_h = f_h$ and from \eqref{eq:disc-Lap}, one obtains
\begin{equation*}
 v_h  =  L_h^{-1} P_h f_h = L_h^{-1} P_h Av = R_h v.
\end{equation*}
Therefore, we have
\begin{equation*}
\| v_h \|_{L^q} \le \| R_h v \|_{W^{1,q}} \le C \| v \|_{W^{1,q}} 
\le C \| v \|_{W^{2,q}} \le C \| Av \|_{L^q} = C \| L_h v_h \|_{L^q}
\end{equation*}
by Lemma \ref{lem:projections} and \eqref{eq:regularity}.
\qed\end{proof}

Furthermore, the following estimate holds. See \cite[Lemma 4.6]{Sai07} for the proof.
\begin{lemma}\label{lem:Kh}
Assume that $\{ \cT_h \}_h$ satisfies (H1) when $q \ne 2$.
Let $q \in (\mu', \mu)$.
Then, there exists $C>0$ depending only on $q$ and $\Omega$ such that
\begin{equation*}
\| A_h^{-1}(I - K_h^{-1}) v_h \|_{h,q} \le C h^2 \| \nabla v_h \|_{L^q},
\quad v_h \in S_h.
\end{equation*}
\end{lemma}

\section{Proofs of Theorems \ref{cor:mr_disc-Laplacian},
 \ref{cor:dmr_disc-Laplacian}, \ref{cor:dmr_disc-Lap_finite} and \ref{cor:dmr_disc-Lap_initial}}
\label{section:proof1}

The aim of this section is to establish CMR and DMR for $A_h$.
We first consider the continuous case via the method of imaginary powers
of operators. 
Then, we obtain DMR for $A_h$ by our previous result (Lemma
\ref{thm:sufficient}). We also present a useful criterion to check the
condition (NR)$_{\delta,\varepsilon}$.

In view of Lemma \ref{lem:mr-imaginary_power}, it suffices to show that 
\begin{equation}
-A_h \in \cP(X_{h,q};K) \cap \bip(X_{h,q};M,\theta)
\end{equation}
for some $K>0$, $M \ge 1$, and $\theta \in [0, \pi/2)$, uniformly with respect to $h$.
We first show that $-A_h \in \cP(X_{h,q};K)$.
\begin{lemma}\label{lem:positivity_disc-Lap}
Let $q \in (1, \mu)$.
Assume that the family $\{ \cT_h \}_h$ satisfies (H1) and (H2) when $q \ne 2$.
Then, there exists $K_q > 0$ satisfying
\begin{equation*}
-A_h \in \cP(X_{h,q}; K_q), 
\end{equation*}
where $K_q$ is independent of $h>0$.
\end{lemma}

\begin{proof}
Let $T_h(t)$ be the semigroup $e^{tA_h}$ generated by $A_h$ in $X_{h,q}$.
Then, by Lemma \ref{prop:disc-Laplacian2}, $T_h(t)$ is an analytic and contraction semigroup.
Since $T_h(t)$ is contraction semigroup, we have
\begin{equation*}
\| R(\lambda; A_h) \|_{\cL(X_{h,q})} \le \frac{1}{\lambda} , \quad \forall \lambda > 0 
\end{equation*}
for each $h>0$.
In addition, by virtue of Lemma \ref{lem:inv_disc-Lap} and analyticity of $T_h(t)$,
we have
\begin{equation*}
\| R(\lambda; A_h) f_h \|_{h,q} 
= \| A_h^{-1} [ \lambda R(\lambda; A_h) - I ] f_h \|_{h,q}
\le C \| f_h \|_{h,q}, \quad \forall f_h \in S_h
\end{equation*}
for all $\lambda > 0$ and $h>0$, where $C>0$ is independent of $h$.
Therefore, we obtain $-A_h \in \cP(X_{h,q};K_q)$ with $K_q=C+1$ 
since $R(\cdot; A_{h,q}) \in C([0,\infty); \cL(X_{h,q}))$.
\qed\end{proof}

To show $-A_h \in \bip(X_{h,q};M,\theta)$, we use Duong's result, which is based on $H^\infty$-functional calculus.
The imaginary power is understood as the special case of the function of operators.
Let $X$ be a Banach space, $D \subset \bC$ be a domain and $\mathcal{O}(D)$ be the space of holomorphic functions on $\bC$.
We set
\begin{equation}
H^\infty(D) = \mathcal{O}(D) \cap L^\infty(D; \bC).
\label{eq:H-infty}
\end{equation}
Then, for $A \in \cP(X)$ and for $m \in H^\infty(\Sigma_\theta)$ with suitable $\theta$,
$m(A)$ can be defined as a linear operator on $X$.
When we take $m(z) = z^{it}$, the imaginary power $A^{it}$ is defined in this sense.
The definition and details of the properties of $m(A)$ have been presented in the literature \cite{CowDMY96} and in the Appendix \ref{section:functional_calculus}.
We refer to \cite{Duo90} for the proof of the following lemma (see also \cite{CoiW76}).

\begin{lemma}[\mb{\cite[Theorem 2]{Duo90}}]
\label{thm:Duong}
Let $(\Omega, \mu)$ be a $\sigma$-finite measure space and let $A$ be a linear operator on $X = L^q(\Omega, \mu)$ for $q \in (1, \infty)$.
Assume that $A \in \cP(X)$ and that $-A$ generates a contraction semigroup $T(t)$ on $X$.
Moreover, we suppose that $T(t)$ is positivity-preserving on $X$.
Then, for each $\theta \in (\pi/2, \pi)$, there exists $M > 0$ satisfying
\begin{equation*}
\| m(A) \|_{\cL(X)} \le M \| m \|_{L^\infty(\Sigma_\theta)}
\end{equation*}
for all $m \in H^\infty(\Sigma_\theta)$.
Furthermore, $M$ depends only on $q$ and $\theta$, but is independent of $A$ and measure space $(\Omega, \mu)$.
\end{lemma}

\begin{lemma}\label{cor:Duong}
Let $X$ and $A$ be as in Lemma \ref{thm:Duong}.
Then, for each $\theta \in (\pi/2, \pi)$, there exists $M>0$ such that $A \in \bip(X;M, \theta)$.
\end{lemma}

\begin{proof}
Let $m(z) = z^{it}$ for $z \in \Sigma_\theta$ and $t \in \bR$.
Here, $z^{it}$ is defined as
\begin{equation*}
z^{it} = e^{it(\log|z| + i \arg z)}, \quad \arg z \in (-\pi, \pi)
\end{equation*}
for $z \in \Sigma_\pi$.
Then, setting $z = |z| e^{i \vartheta}$ ($\vartheta \in (-\theta, \theta)$), one can readily obtain $| z^{it} | = e^{-t \vartheta}$.
Therefore, we have
\begin{equation*}
\| m \|_{L^\infty(\Sigma_\theta)} \le e^{|t|\theta},
\end{equation*}
which yields $m \in H^\infty(\Sigma_\theta)$ and $A \in \bip(M,\theta)$ for some $M>0$ by Lemma \ref{thm:Duong}.
\qed\end{proof}

Now, we are ready to show the following lemma.

\begin{lemma}[Imaginary powers of discrete Laplacian]
\label{thm:bip_disc-Laplacian}
Let $q \in (1, \mu)$.
Assume that (H1) and (H2) are satisfied when $q \ne 2$.
Then there exist $M_q > 0$ and $\theta_q \in (0, \pi/2)$ satisfying
\begin{equation*}
-A_h \in \bip(X_{h,q}; M_q, \theta_q), 
\end{equation*}
where $M_q$ and $\theta_q$ are independent of $h>0$.
\end{lemma}

\begin{proof}
We begin by proving that $-A_h \in \bip(X_{h,q}; M, \theta)$ for each $\theta \in (\pi/2, \pi)$ and for suitable $M>0$ independent of $h$.
Let $T_h(t)$ be the semigroup $e^{tA_h}$ generated by $A_h$ in $X_{h,q}$.
Then, by Lemma \ref{prop:disc-Laplacian1} and \ref{lem:positivity_disc-Lap}, we can apply Lemma \ref{cor:Duong}. Therefore,
for each $\theta \in (\pi/2, \pi)$, there exists $M>0$ satisfying
\begin{equation}
-A_h \in \bip(X_{h,q}; M, \theta) .
\label{eq:imag-power_disc-Lap}
\end{equation}

Now, we show our assertion.
We first assume that $q = 2$.
In this case, $X_{h,2}$ is a Hilbert space and $-A_h$ is self-adjoint
and positive definite without conditions on the triangulation by Poinc\'are inequality.
Consequently, by Theorem \ref{thm:functional_Hilbert},
we have
\begin{equation*}
\| (-A_h)^{it} \|_{\cL(X_{2,h})} \le \int_0^\infty dE_{-A_h}(\lambda) = 1
\end{equation*}
for all $t \in \bR$, which implies $-A_h \in \bip(X_{2,h};1,0)$.
Here, $E_{-A_h}$ is the spectral decomposition of $-A_h$.
Then we presume that $q \ne 2$. Set 
\begin{equation*}
\theta_{q,r} = \frac{q^{-1} - 2^{-1}}{r^{-1} - 2^{-1}}
\end{equation*}
for $r \ne 2$. Since $q \ne 2$, we can choose $r \in (1, \infty)$ satisfying $\theta_{q,r} \in (0,1)$. 
Then, by the Riesz-Thorin theorem, we obtain
\begin{equation*}
\| (-A_h)^{it} \|_{X_{h,q}}
\le \| (-A_h)^{it} \|_{X_{h,2}}^{1 - \theta_{q,r}} 
    \| (-A_h)^{it} \|_{X_{h,r}}^{\theta_{q,r}} 
\le M^{\theta_{q,r}} e^{\theta \theta_{q,r} |t|}
\end{equation*}
for any $t \in \bR$ and $\theta \in (\pi/2, \pi)$, 
where $M>0$ is as in \eqref{eq:imag-power_disc-Lap}.
Since $\theta_{q,r} \in (0,1)$, we can take $\theta$ as
\begin{equation*}
\frac{\pi}{2} < \theta < \frac{\pi}{2 \theta_{q,r}},
\end{equation*}
which implies 
\begin{equation*}
-A_h \in \bip(X_{h,q}; M^{\theta_{q,r}}, \theta \theta_{q,r}) 
\end{equation*}
with $\theta \theta_{q,r} < \pi/2$. 
This is the desired assertion.
\qed\end{proof}

Owing to Lemma \ref{thm:sufficient} and Theorem
\ref{cor:mr_disc-Laplacian}, we are able to obtain DMR for $A_h$.
To apply Lemma \ref{thm:sufficient}, it is necessary to verify that the
condition (NR)$_{\delta, \varepsilon}$ is satisfied.
From Lemma \ref{prop:disc-Laplacian2}, the condition (NR1) is always satisfied.
Therefore, what is left is to check the condition (NR2).
We begin with the following lemma,
which is a generalization of \cite[Lemma 2]{Fuj73}.
No condition on the triangulation is required.

\begin{lemma}\label{lem:inv_ineq}
Let $r \in [1,\infty)$. Then, we have
\begin{equation}
\| \nabla v_h \|_{L^r} \le \frac{d+1}{\kappa_h} \| v_h \|_{h,r} ,\quad \forall v_h \in S_h . \label{eq:inv_ineq}
\end{equation}
\end{lemma}

\begin{proof} 
Fix $K \in \cT_h$ arbitrarily. Then it suffices to show that 
\begin{equation}
\| \nabla v_h \|_{L^r(K)} \le \frac{d+1}{\kappa_h} \| v_h \|_{h,r,K} ,\quad \forall v_h \in S_h , \label{eq:inv_ineq_K}
\end{equation}
where $\| v_h \|_{h,r,K} = \| M_h v_h \|_{L^r(K)}$.
Let $Q_j$ $(j = 0,\dots, d)$ be the vertex of $K$, $\lambda_j$ be the corresponding barycentric coordinate in $K$, and $\kappa_j$ be the length of the perpendicular from $P_j$ in $K$.
Then it is well-known that $|\nabla\lambda_j| = 1/\kappa_j$.
Take $v_h \in S_h$ arbitrarily and set $v_j = v_h(Q_j)$.
Since $v_h|_K = \sum_{j=0}^{d} v_j \lambda_j$, we have
\begin{align}
\| \nabla v_h \|_{L^r(K)} 
& \le \sum_{j=0}^d | v_j | \| \nabla \lambda_j \|_{L^r(K)}  = \sum_{j=0}^d \frac{| u_j |}{\kappa_j} |K|^{1/r} \nonumber \\
& \le \left( \sum_{j=0}^d \frac{1}{\kappa_j^{r'}} \right) ^{1/r'}
 \left( \sum_{j=0}^d |u_j|^r \right) ^{1/r} |K|^{1/r}  \nonumber \\
 & \le \frac{(d+1)^{1/r'}}{\kappa_h} \left( |K| \sum_{j=0}^d |u_j|^r \right) ^{1/r},  \label{eq:inv_ineq1}
\end{align}
where $r'$ is the H\"older conjugate of $r$.
Moreover, it is readily apparent that
\begin{equation*}
\| v_h \|_{h,r,K} = \left( \frac{1}{d+1} |K| \sum_{j=0}^d |v_j|^r \right) ^{1/r}.
\end{equation*}
This, together with \eqref{eq:inv_ineq1}, implies that
\begin{equation*}
\| \nabla v_h \|_{L^r(K)} \le \frac{(d+1)^{1/r' + 1/r}}{\kappa_h} \| v_h \|_{h,r,K}
= \frac{d+1}{\kappa_h} \| v_h \|_{h,r,K}.
\end{equation*}
Thereby we complete the proof.
\qed\end{proof}

Now, we describe a sufficient condition for (NR2) to hold.
\begin{lemma}[A sufficient condition for $\mathrm{(NR)}_{\delta, \varepsilon}$]
\label{prop:NR-suff_cond}
Assume $\theta \in [0,1/2)$ and $q \in (1, \infty)$, and
Let $\theta_q = \arccos|1-2/q|$.
If we choose $\varepsilon$ and $\tau$ sufficiently small so that $A_h$ satisfies
 \eqref{eq:sufficient-NR}, 
for every $h$, then the condition $\mathrm{(NR)}_{\theta_q, \varepsilon}$ is fulfilled.
\end{lemma}

\begin{proof}
The numerical range of $A_h$ is expressed as
\begin{equation}
S(A_h) = \{ (A_h v_h, v_h^*)_h \mid v_h\in S_h,\ \| v_h \|_{h,q} = 1 \},
\end{equation}
where $v_h^* \in S_h$ is defined as  
\begin{equation}
 \label{eq:v*}
 v_h^*(P) = |v_h(P)|^{q-2}v_h(P)\qquad \text{for every node $P$ of $\cT_h$} 
\end{equation}
for $v_h \in S_h$.
Therefore, by Lemma \ref{lem:inv_ineq}, we have
\begin{equation*}
| (A_h v_h, v_h^*)_h | \le \| \nabla v_h \|_{L^q} \| \nabla v_h^* \|_{h,q'}
\le \frac{(d+1)^2}{\kappa_h^2} \| v_h \|_{h,q}^q
\end{equation*}
for all $v_h \in S_h$. 
Hence we can deduce (NR2) form the assumption \eqref{eq:sufficient-NR}.
\qed\end{proof}

\medskip

At this stage, we can state the following proofs. 

\begin{proof}[Proof of Theorem \ref{cor:mr_disc-Laplacian}]
 It is a consequence of Lemmas \ref{lem:mr-interval},
 \ref{cor:mr-imaginary_power}, and \ref{thm:bip_disc-Laplacian}. 
\qed\end{proof}

\begin{proof}[Proof of Theorem \ref{cor:dmr_disc-Laplacian}]
 It is a consequence of Theorem \ref{cor:mr_disc-Laplacian} and
 Lemmas \ref{thm:sufficient} and \ref{prop:NR-suff_cond}.
\qed\end{proof}

\begin{proof}[Proof of Theorem \ref{cor:dmr_disc-Lap_finite}]
 It is a consequence of Theorem \ref{cor:dmr_disc-Laplacian} and
 Lemma \ref{lem:dmr_finite}.
\qed\end{proof}

\begin{proof}[Proof of Theorem \ref{cor:dmr_disc-Lap_initial}]
It is a consequence of Theorem \ref{cor:dmr_disc-Laplacian} and Lemma \ref{lem:dmr_initial}.
 \qed\end{proof}

\section{Proof of Theorem \ref{thm:error_linear}}
\label{sec:lin}

This section is devoted to error analysis of 
the solution $u_h = (u_h^n) \in l^p(N_T;S_h)$ of
\eqref{eq:disc-heat}. 
We begin by presenting some lemmas.

\begin{lemma}\label{lem:Sobolev}
Let $X$ be a Banach space, $T>0$, $p \in (1, \infty)$ and $\tau \in (0,1)$. 
Set $t_n = n \tau$ for $n = 0, 1, \dots, N_T$.
Then, there exists $C_\mathrm{S} > 0$ satisfying
\begin{equation}
\left( \sum_{n=0}^{N_T-1}\| v(t_n) \|_X^p \tau \right)^{1/p} 
+\left( \sum_{n=1}^{N_T}\| v(t_n) \|_X^p \tau \right)^{1/p} 
\le C_\mathrm{S} \| v \|_{W^{1,p}(J_T; X)}
\label{eq:Sobolev}
\end{equation}
for all $v \in W^{1,p}(J_T; X)$, where $C_\mathrm{S}$ depends only on $p$, but is independent of $T$, $\tau$, and $X$.
\end{lemma}

\begin{proof}
By the Sobolev embedding $W^{1,p}(0,1; X) \hookrightarrow L^\infty(0,1; X)$, there exists $C_1 > 0$ such that
\begin{equation*}
\| v \|_{L^\infty(0,1; X)} \le C_1 \| v \|_{W^{1,p}(0,1; X)}
\end{equation*}
for $v \in W^{1,p}(0,1; X)$. 
One can check that $C_1$ is independent of $X$. 
See the proof of \cite[Theorem 8.8]{Bre11}.
Then, setting $J_n = (t_n, t_{n+1})$ and considering the change of variables, we have
\begin{equation*}
\| v(t_n) \|_X \le \| v \|_{L^\infty(J_n; X)} 
\le C_1 (1 + \tau) \tau^{-1/p} \| v \|_{W^{1,p}(J_n; X)}
\end{equation*}
for each $n \in \bN$.
Therefore, we have \eqref{eq:Sobolev} with $C_\mathrm{S} = 2C_1$.
\qed\end{proof}

The next lemma is shown readily by Taylor's theorem. Therefore, we skip the proof.
\begin{lemma}\label{lem:fdm}
Let $X$ be a Banach space, $T>0$, $p \in (1, \infty)$, $\theta \in [0,1]$ and $\tau \in (0,1)$. 
Set $t_n = n \tau$ for $n = 0,1,\dots,N_T$ and 
\begin{equation*}
r^n = \frac{v(t_{n+1}) - v(t_n)}{\tau} - \left[ (1- \theta) \frac{dv}{dt}(t_n) + \theta \frac{dv}{dt}(t_{n+1}) \right]
\end{equation*}
for $v \in W^{j_\theta+1, p}(J_T; X)$, where $j_\theta$ is defined as \eqref{eq:jt}.
Then, there exists $C>0$ such that
\begin{equation*}
\left( \sum_{n=0}^{N_T-1} \| r^n \|_X^p \tau \right)^{1/p} 
\le C \tau^{j_\theta} \| v \|_{W^{j_\theta+1, p}(J_T; X)},
\end{equation*}
where $C$ is independent of $\tau$ and $X$.
\end{lemma}

\medskip

Now we can state the following proof. 

\begin{proof}[Proof of Theorem \ref{thm:error_linear}]
We set $e_h^n = u_h^n - P_h U^n $ so that
\begin{equation*}
u_h^n - U^n = e_h^n + (P_h - I)U^n.
\end{equation*}
Then, by Lemmas \ref{lem:projections} and \ref{lem:Sobolev}, we have
\begin{align}
&\sum_{n=0}^{N_T-1} \| (P_h - I)U^{n + \theta} \|_{L^q}^p \tau \\
& \le C h^{2p} \left[ (1-\theta)^p \sum_{n=0}^{N_T-1} \| U^n
 \|_{W^{2,q}}^p \tau + \theta^p \sum_{n=1}^{N_T} \| U^n \|_{W^{2,q}}^p
 \tau \right] \nonumber \\
& \le C h^{2p} \| u \|_{W^{1,p}(J_T; W^{2,q})}^p .
\label{eq:error1}
\end{align}

It remains 
to derive an estimation for $e_h^n$.
Set $V^n = \partial_t u(\cdot, t_n)$ and 
\begin{equation*}
r_h^{n, \theta} 
= (K_h^{-1} P_h A - A_h P_h)U^{n+\theta} 
+ P_h \left( \frac{u(t_{n+1}) - u(t_n)}{\tau} \right)
- K_h^{-1}P_h V^{n+\theta} .
\end{equation*}
Then, by a simple computation, we have
\begin{equation*}
\begin{cases}
(D_\tau e_h)^n = A_h e_h^{n+\theta} + r_h^{n, \theta}, & n =0,1,\dots,N_T-1, \\
e_h^0 = 0.
\end{cases}
\end{equation*}
Therefore,
\begin{equation*}
\begin{cases}
(D_\tau (A_h^{-1}e_h) )^n = A_h (A_h^{-1} e_h^{n+\theta}) + A_h^{-1} r_h^{n, \theta} , & n =0,1,\dots,N_T-1, \\
A_h^{-1} e_h^0 = 0 .
\end{cases}
\end{equation*}
Consequently, according to Theorem \ref{cor:dmr_disc-Lap_finite}, we obtain
\begin{equation}
 \sum_{n=0}^{N_T-1} \| e_h^{n + \theta} \|_{L^q}^p \tau 
 = \sum_{n=0}^{N_T-1} \| A_h (A_h^{-1} e_h^{n + \theta} ) \|_{L^q}^p
 \tau 
 \le C \sum_{n=0}^{N_T-1} \| A_h^{-1} r_h^{n , \theta}  \|_{L^q}^p \tau.
\label{eq:error2}
\end{equation}
We divide $r_h^{n, \theta}$ into two parts as
\begin{equation*}
r_h^{n, \theta} = r_{1,h}^{n, \theta} + r_{2,h}^{n, \theta},
\end{equation*}
where
\[
  r_{1,h}^{n, \theta} 
 = (K_h^{-1} P_h A - A_h P_h)U^{n+\theta} , ~~
r_{2,h}^{n, \theta} 
 = P_h \left( \frac{u(t_{n+1}) - u(t_n)}{\tau} \right)
- K_h^{-1}P_h V^{n+\theta}.
\]
We first estimate $r_{1,h}^{n,\theta}$.
Noting the relation \eqref{eq:disc-Lap}, we have
\begin{equation*}
A_h^{-1} r_{1,h}^{n,\theta} = (R_h - P_h)U^{n+\theta},
\end{equation*}
so that
\begin{align}
\left( \sum_{n=0}^{N_T-1} \| A_h^{-1} r_{1,h}^{n,\theta} \|_{L^q}^p \tau \right)^{1/p}
& \le Ch^2 \left( \sum_{n=0}^{N_T-1} \| U^{n+\theta} \|_{W^{2,q}}^p \tau \right)^{1/p} \\
& \le C h^2\|u\|_{W^{1,p}(J_T;W^{2,q})}\label{eq:error3}
\end{align}
by Lemma \ref{lem:projections} and Lemma \ref{lem:Sobolev}.
Also, $A_h^{-1} r_{h,2}^{n,\theta}$ is expressed as
\begin{equation*}
A_h^{-1} r_{2,h}^{n,\theta} = 
A_h^{-1} P_h \left[ \frac{u(t_{n+1}) - u(t_n)}{\tau} - V^{n+\theta} \right]
+ A_h^{-1}(I - K_h^{-1})P_h V^{n+\theta}.
\end{equation*}
According to Lemmas \ref{lem:projections}, \ref{lem:inv_disc-Lap},
  \ref{lem:Kh}, \ref{lem:Sobolev}, and \ref{lem:fdm}, we have
\begin{align}
&\left( \sum_{n=0}^{N_T-1} \| A_h^{-1} r_{2,h}^{n,\theta} \|_{L^q}^p \tau \right)^{1/p} \\
& \le C \tau^{j_\theta} \| u \|_{W^{j_\theta+1, p}(J_T; L^q)}
   +  C h^2 \left( \sum_{n=0}^{N_T-1} \| \nabla P_h V^{n + \theta}
 \|_{L^q}^p \tau \right)^{1/p} \nonumber \\
& \le C \tau^{j_\theta}\| u \|_{W^{j_\theta+1, p}(J_T; L^q)} + C h^2 \left( \sum_{n=0}^{N_T-1} \| V^{n}
 \|_{W^{1,q}}^p \tau \right)^{1/p} \nonumber \\
 & \le C \tau^{j_\theta}\| u \|_{W^{j_\theta+1, p}(J_T; L^q)} +
 Ch^2\| \partial_tu \|_{W^{1,p}(J_T; W^{1,q})}. \label{eq:error4}
\end{align}

Combining \eqref{eq:error1}, \eqref{eq:error2}, \eqref{eq:error3}, and \eqref{eq:error4}, we obtain the error estimate \eqref{eq:error_estimate}.
\qed\end{proof}

\section{Proofs of Theorems \ref{thm:sl} and \ref{thm:infty}}
\label{section:app_semilinear}

This section is devoted to analysis of semilinear problems 
\eqref{eq:semilinear} and \eqref{eq:approx-semilinear}. We first prove
several auxiliary lemmas. 
  
\subsection{Embedding and trace theorems}\label{subsection:embedding}
For $q \in (1, \infty)$, we recall that $A_q$ denotes the realization of
the Dirichlet Laplacian defined as \eqref{eq:delq}. 
Let $D(A_q)$ be a Banach space equipped with the norm $\| A_q \cdot \|_{L^q}$.
This is a norm if $q \in (1, \mu)$ by the regularity assumption \eqref{eq:regularity}.
We also set $D(A_{h,q}) = (S_h, \| A_h \cdot \|_{h,q})$, which is a Banach space for $q \in (1, \mu)$ by Lemma \ref{lem:inv_disc-Lap}.

For $N \in \bN \cup \{\infty\}$ and $v_h \in S_h^{N+1}$, we set
\begin{equation}
\label{eq:YT1}
\| v_h \|_{Y^{p,q}_{h,\tau,N}}
= \| v_{h,1} \|_{l^p_\tau(N;X_{h,q})} + \| A_h v_{h,1} \|_{l^p_\tau(N;X_{h,q})} + \| D_\tau v_h \|_{l^p_\tau(N;X_{h,q})}
\end{equation}
and $Y^{p,q}_{h,\tau,N} = \left(S_h^{N+1}, \| \cdot \|_{Y^{p,q}_{h,\tau,N}} \right)$.
For abbreviation, we write $Y^{p,q}_{h,\tau} = Y^{p,q}_{h,\tau,\infty}$ and
\begin{equation}
\label{eq:YT2}
\| v_h \|_{Y_T} = \| v_h \|_{Y^{p,q}_{h,\tau,N_T}}
\end{equation}
for $T>0$, where $N_T$ is defined as \eqref{eq:NT}. 

Then, we have the following embedding result.
\begin{lemma}\label{prop:embedding}
Let $q \in (\mu_d, \mu)$ and $p > 2q/(2q-d)$.
Assume that the family $\{ \cT_h \}_h$ satisfies (H1) and (H2) when $q \ne 2$.
Then, the embedding
\begin{equation*}
(X_{h,q}, D(A_{h,q}))_{1-1/p, p} \inj L^\infty
\end{equation*}
holds uniformly for $h>0$.
\end{lemma}

To show Lemma \ref{prop:embedding}, we prove the discrete Gagliardo-Nirenberg type inequality. 
The following result is the generalization of 
\cite[Lemma 3.3]{Han02}, and that the proof is almost identical. 
However, for the reader's convenience, we provide the proof.

\begin{lemma}[Discrete Gagliardo--Nirenberg type inequality]\label{lem:disc-GN}
Let $q \in (\mu_d,\mu)$.
Assume that the family $\{ \cT_h \}_h$ satisfies (H1) and (H2).
Then, we have
\begin{equation}
\| v_h \|_{L^\infty} \le C \| A_h v_h \|_{h,q}^{\frac{d}{2q}} \| v_h \|_{h,q}^{1- \frac{d}{2q}},
\quad \forall v_h \in S_h.
\label{eq:disc-GN}
\end{equation}
\end{lemma}

\begin{proof}
It suffices to show that 
\begin{equation}
\| L_h^{-1} f_h \|_{L^\infty} \le C 
\| f_h \|_{L^q}^{\frac{d}{2q}} \| L_h^{-1} f_h \|_{L^q}^{1-\frac{d}{2q}},
\label{eq:GN0}
\end{equation}
for every $f_h \in S_h$.
We decompose the left-hand side as
\begin{equation}
\| L_h^{-1} f_h \|_{L^\infty}
\le \| ( L_h^{-1} - P_h A_q^{-1} f_h )f_h \|_{L^\infty}
+ \| P_h A_q^{-1} f_h f_h \|_{L^\infty}
=: a + b.
\label{eq:GN1}
\end{equation}
From the usual Gagliardo-Nirenberg inequality \cite[Theorem 5.9]{AdaF03} and the regularity assumption \eqref{eq:regularity}, we have
\begin{align}
b \le C \|A_q^{-1} f_h \|_{L^\infty} 
\le& C \| f_h\|_{L^q}^{\frac{d}{2q}} \| A_q^{-1} f_h
 \|_{L^q}^{1-\frac{d}{2q}} \nonumber \\
\le& C \| f_h\|_{L^q}^{\frac{d}{2q}} \left(
\| L_h^{-1} f_h \|_{L^q}^{1-\frac{d}{2q}}
+ 
\| (A_q^{-1} - L_h^{-1}P_h) f_h \|_{L^q}^{1-\frac{d}{2q}}
\right).
\label{eq:GN2}
\end{align}
Setting $u = A_q^{-1}f_h \in D(A_q)$, we have
\begin{equation}
\| (A_q^{-1} - L_h^{-1}P_h) f_h \|_{L^q}
= \| u - R_h u \|_{L^q}
\le Ch^2 \|u\|_{W^{2,q}}
\le Ch^2 \| f_h \|_{L^q},
\label{eq:GN3}
\end{equation}
by Lemma \ref{lem:projections} and \eqref{eq:regularity}.
Since Lemma \ref{lem:inv_ineq} and the inverse assumption imply
\begin{equation*}
\| L_h v_h \|_{L^q} \le C h^{-2} \| v_h \|_{L^q}, 
\quad \forall v_h \in S_h,
\end{equation*}
we obtain
\begin{equation}
\| (A_q^{-1} - L_h^{-1}P_h) f_h \|_{L^q} \le 
C \| L_h^{-1} f_h \|_{L^q}.
\label{eq:GN4}
\end{equation}
From \eqref{eq:GN2} and \eqref{eq:GN4}, we have
\begin{equation}
b \le C \| f_h \|_{L^q}^{\frac{d}{2q}} \| L_h^{-1} f_h \|_{L^q}^{1-\frac{d}{2q}}.
\label{eq:GN5}
\end{equation}
 We estimate $a$. The inverse assumption (H1) is well known to imply (see \cite[theorem 3.2.6]{Cia78}) the inverse inequality
\begin{equation*}
\| v_h \|_{L^\infty} \le C h^{-d/r} \| v_h \|_{L^q},
\quad \forall v_h \in S_h,
\end{equation*}
 where $C>0$ is independent of $h$.
This, together with \eqref{eq:GN3} and \eqref{eq:GN4}, implies
\begin{align*}
a 
= \| P_h (L_h^{-1}P_h - A_q^{-1}) f_h \|_{L^\infty} & \le C h^{-d/q} \| (L_h^{-1}P_h - A_q^{-1}) f_h \|_{L^q} \\
&\le C \| f_h \|_{L^q}^{\frac{d}{2q}} \| L_h^{-1} f_h \|_{L^q}^{1-\frac{d}{2q}}.
\end{align*}
Therefore, we can complete the proof.
\qed\end{proof}

\begin{proof}[Proof of Lemma \ref{prop:embedding}]
From the general theory of interpolation spaces, it is readily apparent that the embedding
\begin{equation*}
(X_{h,q}, D(A_{h,q}))_{1-1/p, p} \inj (X_{h,q}, D(A_{h,q}))_{1-1/p-\varepsilon, 1}
\end{equation*}
for $\varepsilon \in (0,1-1/p)$, uniformly with respect to $h$. 
Take $\varepsilon = 1-1/p - d/(2q)$ so that $1-1/p-\varepsilon = d/(2q)$.
Then, the assumptions $q > d/2$ and $p > 2q/(2q-d)$
imply $\varepsilon \in (0,1-1/p)$.
Therefore, we can obtain from Lemma \ref{lem:disc-GN} that the embedding
\begin{equation*}
(X_{h,q}, D(A_{h,q}))_{d/(2q), 1} \inj L^\infty,
\end{equation*}
holds uniformly with respect to $h$, by the same argument of the embedding theorem for the Besov spaces (see \cite[Theorem 7.34]{AdaF03}).
\qed\end{proof}

We next show the trace theorem for $Y^{p,q}_{h,\tau}$.
The following result is the discrete version of the characterization of the real interpolation space via the analytic semigroup \cite[Lemma 6.2]{Lun09}.
\begin{lemma}\label{prop:trace-infinite}
Let $q \in (1, \mu)$ and $p \in (1,\infty)$.
Assume that the family $\{ \cT_h \}_h$ satisfies (H1) and (H2) when $q \ne 2$.
Then
there exists $C>0$ depending only on $p$ such that
\begin{equation*}
\sup_{n \ge 1} \| v_h^n \|_{1-1/p,p} \le C \| v_h \|_{Y^{p,q}_{h,\tau}}.
\end{equation*}
for every $v_h \in Y^{p,q}_{h,\tau}$.
\end{lemma}

\begin{proof}
Fix $v_h \in Y^{p,q}_{h,\tau}$ arbitrarily.
It suffices to show that
\begin{equation}
\| v_h^1 \|_{1-1/p,p} \le C \| v_h \|_{Y^{p,q}_{h,\tau}}
\label{eq:trace1}
\end{equation}
by translation. 
Since
\begin{equation*}
v_h^1 = - \sum_{j=1}^{n} (v_h^{j+1} - v_h^j) + v_h^{n+1}
\end{equation*} 
for $n \ge 1$, we have
\begin{equation}
K(t, v_h^1) \le \sum_{j=1}^{n} \| v_h^{j+1} - v_h^j \|_{h,q} + t \| A_h v_h^{n+1} \|_{h,q}
\label{eq:trace2}
\end{equation}
for $t > 0$.
Here, the function
\begin{multline}
K(t,w_h) = \inf \{ \|a_h\|_{h,q} + t \|A_h b_h\|_{h,q}
\mid w_h = a_h + b_h,\ a_h, b_h \in X_{h,q}. \} , \\
\quad t > 0 , \ w_h \in X_{h,q}
\end{multline}
is the $K$-function with respect to the interpolation pair $(X_{h,q}, D(A_{h,q}))$ (see \cite{Lun09} and \cite{Tri95}).
Then, \eqref{eq:trace2} implies that
\begin{align}
\| v_h^1 \|_{1-1/p, p}^p 
&= \int_0^\infty \left|t^{-1+1/p} K(t, v_h^1) \right|^p \frac{dt}{t}\nonumber  \\
&\le \int_0^\tau | t^{-1}K(t,v_h^1)|^p dt \nonumber \\
&+ 2^p \sum_{n=1}^{\infty} \left[ 
\int_{n\tau}^{(n+1)\tau} \left( \frac{1}{t} 
\sum_{j=1}^{n} \| v_h^{j+1} - v_h^j \|_{h,q} \right)^p
dt + 
\tau \| A_h v_h^{n+1} \|_{h,q}^p \right]\nonumber  \\
&\le 2^p \| A_h v_h \|_{l^p_\tau(\bN; X_{h,q})}^p
  + 2^p \sum_{n=1}^{\infty} I_n.
\label{eq:trace3}
\end{align}
In the last step, we used the property $K(t, v_h^1) \le t \| A_h v_h^1 \|_{h,q}$ and we defined $I_n$ as 
\begin{equation*}
I_n = \int_{t_n}^{t_{n+1}} \left( \frac{1}{t} 
\sum_{j=1}^{n} \| v_h^{j+1} - v_h^j \|_{h,q} \right)^p
dt
\end{equation*}
for $n \ge 1$.
The term $I_n$ is bounded as
\begin{equation*}
I_n \le \int_{t_n}^{t_{n+1}} \left( \frac{1}{n\tau} 
\sum_{j=1}^{n} \| v_h^{j+1} - v_h^j \|_{h,q} \right)^p
dt
= \tau \left( \frac{1}{n} \sum_{j=1}^{n} \| (D_\tau v_h)^j \|_{h,q} \right)^p.
\end{equation*}
Therefore, we can obtain
\begin{equation}
\sum_{n=1}^{\infty} I_n 
\le \left( \frac{p}{p-1} \right)^p 
\sum_{n=1}^{\infty} \| (D_\tau v_h)^j \|_{h,q}^p \tau
\label{eq:trace4}
\end{equation}
by the Hardy inequality \cite{HarLP52}, and inequalities
\eqref{eq:trace3} and \eqref{eq:trace4} imply \eqref{eq:trace1}, with a constant $C$ depending only on $p$.
\qed\end{proof}

For $Y^{p,q}_{h,\tau,N}$, we have the following trace theorem.
\begin{lemma}\label{prop:trace-finite}
Let $N \in \bN$, $q \in (1, \mu)$ and $p \in (1,\infty)$.
Assume that the family $\{ \cT_h \}_h$ satisfies (H1) and (H2) when $q \ne 2$.
Then, there exists $C>0$ independent of $N$, $h$, and $\tau$ such that
\begin{equation}
\max_{0 \le n \le N} \| v_h^n \|_{1-1/p,p} 
\le C \left( \| v_h \|_{Y^{p,q}_{h,\tau,N}} + \| v_h^0 \|_{1-1/p,p} \right)
\end{equation}
for every $v_h \in Y^{p,q}_{h,\tau,N}$.
\end{lemma}

To prove this result, 
we need to extend each element of $Y^{p,q}_{h,\tau,N}$
to that of $Y^{p,q}_{h,\tau,\infty}$.
First, we obtain the following extension lemma,
which corresponds to \cite[Lemma 7.2]{Ama05}.
\begin{lemma}\label{lem:extension}
Let $X$ be a Banach space and $A$ be a linear operator which has discrete maximal regularity and which satisfies $0 \in \rho(A)$.
Let $N \in \bN \cup \{\infty\}$ and set
\begin{equation*}
\| v \|_{p,N} = \| v_1 \|_{l^p_\tau(N;X)} + \| Av_1 \|_{l^p_\tau(N;X)} + \| D_\tau v \|_{l^p_\tau(N;X)}
\end{equation*}
for $v \in X^{N+1}$ and $Y^p_N = \{ v \in X^{N+1} \mid v^0 \in (X, D(A))_{1-1/p,p} ,\ \| v \|_{p,N} < \infty \}$.
Then, for $M \in \bN$ with $M < N$, there exists a map 
$\ext_M \colon Y^p_M \to Y^p_N$ satisfying
\begin{equation*}
(\ext_M v)^n = v^n, \quad n = 0, \dots M,
\end{equation*}
and
\begin{equation*}
\| \ext_M v \|_{p,N} \le C \left( \| v \|_{p,M} + \| v^0 \|_{1-1/p, p} \right),
\end{equation*}
where $C$ is independent of $\tau$ and $M$.
\end{lemma}

\begin{proof}
For $v \in Y^p_M$, we define $g \in Y^p_N$ as
\begin{equation*}
g^n = \begin{cases}
(D_\tau v)^n - A v^{n+1} , & n=0,\dots, N-1, \\
0 , & \text{otherwise.}
\end{cases}
\end{equation*}
Let $V$ be the solution of 
\begin{equation*}
\begin{cases}
(D_\tau V)^n = V^{n+1} + g^n, & n \in \bN, \\
V^0 = v^0,
\end{cases}
\end{equation*}
which is uniquely solvable by discrete maximal regularity of $A$.
Then, if we set $\ext_M v = V$, it satisfies the desired properties.
Indeed, since $w^n = v^n - V^n$ satisfies
\begin{equation*}
\begin{cases}
(D_\tau w)^n = Aw^n, & n = 0, \dots, M-1, \\
w^0 = 0,
\end{cases}
\end{equation*}
we can obtain $w^n = (I - \tau A)^{-n} w^0 = 0$ for $n=0,\dots, M$.
Moreover, by discrete maximal regularity, we have
\begin{equation*}
\| V \|_{p,N} 
\le C \left( \| g \|_{l^p_\tau(N;X)} + \| V^0 \|_{1-1/p, p} \right) 
\le C \left( \| v \|_{p,M} + \| v^0 \|_{1-1/p, p} \right).
\end{equation*}
\qed\end{proof}

\begin{proof}[Proof of Lemma \ref{prop:trace-finite}.]
Let $v_h \in Y^{p,q}_{h,\tau,N}$. 
Then, by Lemmas \ref{prop:trace-infinite} and \ref{lem:extension},
we have
\begin{align*}
\max_{0 \le n \le N} \| v_h^n \|_{1-1/p,p} 
&\le \sup_{n \in \bN} \| ( \ext_N v_h )^n \|_{1-1/p,p} \\
&\le C \left( \| \ext_N v_h \|_{Y^{p,q}_{h,\tau}} + \| v_h^0 \|_{1-1/p,p} \right)  \\
&\le C \left( \| v_h \|_{Y^{p,q}_{h,\tau,N}} + \| v_h^0 \|_{1-1/p,p} \right) 
\end{align*}
\qed\end{proof}

\subsection{Fractional powers}
\label{sec:fp}

We will use the fractional power $(-A_h)^z$ for $z \in (0,1)$ and $z \in
(-1,0)$; see \cite{Paz83}. The negative powers are defined as
\begin{equation}
(-A_h)^{-z}v_h 
= \frac{\sin(\pi z)}{\pi}
\int_0^\infty t^{-z} R(t; A_h) v_h dt
\label{eq:def_fractional}
\end{equation}
for $z \in (0,1)$. Since $-A_h$ is an operator of positive type, it is
well-defined. One can check that $(-A_h)^{-z}$ is invertible.
Consequently, the positive power $(-A_h)^z$ defined by the inverse operator of $(-A_h)^{-z}$ for $z \in (0,1)$.
Fractional powers satisfy the following interpolation properties:
\begin{align}
\| (-A_h)^z v_h \|_{h,q} 
&\le C \| v_h \|_{h,q}^{1-z} \| A_h v_h \|_{h,q}^z, 
\label{eq:interpolation-1} \\
\| (-A_h)^{-z} v_h \|_{h,q} 
&\le C \| v_h \|_{h,q}^{1-z} \| A_h^{-1} v_h \|_{h,q}^z,
\label{eq:interpolation-2}
\end{align}
for each $z \in (0,1)$ and $v_h \in S_h$, uniformly for $h$.
Consequently, we have
\begin{equation}
\| (-A_h)^{-z} v_h \|_{h,q} \le C \| v_h \|_{h,q},
\quad \forall v_h \in S_h
\label{eq:stability_fractional}
\end{equation}
uniformly for $h$, because of Lemma \ref{lem:inv_disc-Lap}.
Below we set $(-A_h)^0=I$ and $(-A_h)^1=-A_h$. 

\begin{lemma}[Discrete Sobolev inequality]
\label{lem:disc-Sobolev}
Assume that the family $\{ \cT_h \}_h$ satisfies (H1) and (H2).
For every $q > d/2$ and $\alpha \in (d/(2q), 1)$,
there exists $C>0$ independent of $h$, which fulfills the inequality
\begin{equation*}
\| v_h \|_{L^\infty} \le C \| (-A_h)^\alpha v_h \|_{L^q},
\end{equation*}
for all $v_h \in S_h$.
\end{lemma}

\begin{proof}
It suffices to show that 
\begin{equation}
\| (-A_h)^{-\alpha} f_h \|_{L^\infty} 
\le C \| f_h \|_{h,q},
\quad \forall f_h \in S_h.
\label{eq:disc-Sobolev-1}
\end{equation}
By the definition \eqref{eq:def_fractional}, it is necessary to estimate $\| R(t; A_h) f_h \|_{L^\infty}$.
Lemmas \ref{lem:disc-GN} and \ref{lem:positivity_disc-Lap} imply
\begin{equation*}
\| R(t; A_h) f_h \|_{L^\infty}
\le C (1+t)^{-1+\frac{d}{2q}} \| f_h \|_{h,q}.
\end{equation*}
Consequently,
\begin{align}
\| (-A_h)^{-\alpha} f_h \|_{L^\infty} 
& \le \frac{\sin(\pi \alpha)}{\pi}
\int_0^\infty t^{-\alpha} \| R(t; A_h) f_h \|_{L^\infty} dt \nonumber \\
& \le C \int_0^\infty t^{-\alpha} (1+t)^{ -1 + \frac{d}{2q}} dt \| f_h \|_{h,q}.
\label{eq:disc-Sobolev-2}
\end{align} 
Since $\alpha \in (d/(2q), 1)$, the integral in the right-hand-side of \eqref{eq:disc-Sobolev-2} is finite.
Therefore, we can obtain the estimate \eqref{eq:disc-Sobolev-1}.
\qed\end{proof}

\begin{lemma}\label{lem:fractional-embedding}
Assume that the family $\{ \cT_h \}_h$ satisfies (H1) and (H2) when $q \ne 2$.
For every $\beta \in (0, 1-1/p)$, there exists $C>0$ independent of $h$, which satisfies
\begin{equation*}
\| (-A_h)^\beta v_h \|_{h,q} \le C \| v_h \|_{1-\frac{1}{p},p},
\end{equation*}
for all $v_h \in S_h$.
Here, the norm $\| \cdot \|_{1-\frac{1}{p}, p}$ is that of $(X_{h,q}, D(A_h))_{1-\frac{1}{p}, p}$.
\end{lemma}

\begin{proof}
By the general embedding theorem for positive operators \cite[Proposition 4.7]{Lun09}, we have
\begin{equation}
(X_{h,q}, D(A_h))_{\beta, 1} \inj D((-A_h)^\beta).
\label{eq:general-embedding}
\end{equation}
Moreover, $\beta < 1-1/p$ implies
\begin{equation*}
(X_{h,q}, D(A_h))_{1-\frac{1}{p}, p} \inj (X_{h,q}, D(A_h))_{\beta, 1}.
\end{equation*}
Chasing the constants in these proofs, one can show that both embedding properties are uniform for $h$.
Therefore, we can establish the desired estimate.
\qed\end{proof}

\begin{lemma}\label{lem:trace-fractional}
Assume that the family $\{ \cT_h \}_h$ satisfies (H1) and (H2).
For every $\alpha \in (0, \alpha_{p,q,d})$, there exists $C>0$ independent of $h$, which satisfies
\begin{equation*}
\max_{0 \le n \le N} \| v_h \|_{L^\infty}
\le C \left( \| (-A_h)^{-\alpha} v_h \|_{Y^{p,q}_{h,\tau,N}} 
+ \| (-A_h)^{-\alpha} v_h^0 \|_{1-\frac{1}{p},p} \right),
\end{equation*}
for all $N \in \bN$ and $v_h \in S_h^{N+1}$.
\end{lemma}

\begin{proof}
Since $\alpha + d/(2q) < 1 - 1/p$, we can find $\beta \in (0,1)$ that satisfies
\begin{equation*}
\frac{d}{2q} + \alpha < \beta < 1 + \alpha
\quad \text{and} \quad
0 < \beta < 1 - \frac{1}{p}.
\end{equation*}
$\beta - \alpha \in (d/(2q), 1)$.
Then, owing to Lemmas \ref{lem:disc-Sobolev}, \ref{lem:fractional-embedding}, and \ref{prop:trace-finite}, we have
\begin{align*}
\| v_h^n \|_{L^\infty} 
&\le C \| (-A_h)^{\beta - \alpha} v_h^n \|_{h,q}
\le C \| (-A_h)^{-\alpha} v_h^n \|_{1-\frac{1}{p},p} \\
&\le C \left( \| (-A_h)^{-\alpha} v_h \|_{Y^{p,q}_{h,\tau,N}} 
+ \| (-A_h)^{-\alpha} v_h^0 \|_{1-\frac{1}{p},p} \right),
\end{align*}
for $v_h = (v_h^n)_n \in S_h^{N+1}$ and $n \in \bN$.
\qed\end{proof}

\subsection{Completion of the proofs of Theorems \ref{thm:sl} and \ref{thm:infty}}
\label{sec:pf}

Let $u$ and $u_h = (u_h^n)_{n=0}^{N_T}$ be solutions of
  \eqref{eq:semilinear} and \eqref{eq:approx-semilinear}, respectively. Set
  $U^n=u(n\tau)$. We consider the error $e_h = (e_h^n)_{n=0}^{N_T}\in S_h^{N_T+1}$
  defined as 
\[
 {e}_h^n = {u}_h^n - P_h {U}^n\quad (n=0,1,\ldots,N_T).
\]
We first state the sub-optimal error estimate for a globally Lipschitz
nonlinearity 
$f$. If $f$ is a globally Lipschitz continuous function, then \eqref{eq:semilinear} admits a unique time-global solution and the solution of
 \eqref{eq:approx-semilinear} is bounded from above uniformly in $h$ and $\tau$
 (see Remark \ref{rem:bound}). 
Recall that $\|\cdot\|_{Y_T}$ is defined as \eqref{eq:YT1} and
\eqref{eq:YT2}. 

\begin{lemma}
\label{thm:global_error-truncated}
In addition to hypotheses of Theorem \ref{thm:sl}, we assume that $f$ is a globally Lipschitz continuous function.
Then, for every $\alpha \in [0,1]$ and $T \in (0, \infty)$, 
\begin{equation}
\label{eq:frac11}
\| (-A_h)^{-\alpha}{e}_h \|_{Y_{T}}
\le C(h^{2\alpha} + \tau).
\end{equation}
\end{lemma}

\begin{proof}
The proof is divided into two steps. 

\noindent \textbf{Step 1. } We prove that there exists
 $T_1=T_1(u_0,T)\in (0,T)$ satisfying
\begin{equation}
\| (-A_h)^{-\alpha} {e}_{h} \|_{Y_{T_1}} \le C ( h^{2\alpha} + \tau ).
\label{eq:frac12}
\end{equation}
The error ${e}_h$ satisfies
\begin{equation*}
\begin{cases}
(D_\tau {e}_h)^n = A_h {e}_h^{n+1} + {r}_h^n, & n = 0,1,\dots, \\
{e}_h^0 = 0,
\end{cases}
\end{equation*}
where ${r}_h^n = {F}_h({u}_h^n) - P_h (D_\tau {U})^n + A_h P_h {U}^{n+1}$
and ${F}_h = K_h^{-1} \circ P_h \circ {f}$.
We decompose ${r}_h^n$ into two parts:
\begin{equation*}
{r}_h^n = {r}_{1,h}^n + {r}_{2,h}^n,
\quad {r}_{1,h}^n = {F}({u}_h^n) - {F}(P_h {U}^n),
\quad {r}_{2,h}^n = {r}_h^n - {r}_{1,h}^n.
\end{equation*}

We perform an estimation for ${r}_{2,h}^n$. 
Let $V^n = \partial_t {u}(\cdot, n\tau)$. Noting that ${V}^{n+1} =
 A{U}^{n+1} + {f}(U^{n+1})$, the residual term ${r}_{2,h}^n$ is can be decomposed as
\begin{align*}
{r}_{2,h}^n &= R_{1,h}^n + R_{2,h}^n + R_{3,h}^n, \\
R_{1,h}^n &= A_h (P_h - R_h){U}^{n+1}, \\
R_{2,h}^n &= (K_h^{-1} - I)P_h {V}^{n+1} + P_h({V}^{n+1} - (D_\tau {U})^n), \\
R_{3,h}^n &= [ {F}_h(P_h {U}^n) - {F}_h({U}^n) ] + [ {F}_h({U}^n) - {F}_h({U}^{n+1}) ].
\end{align*}
From the interpolation property \eqref{eq:interpolation-1} and the inverse inequality,
we have
\begin{equation*}
\| (-A_h)^\gamma v_h \|_{h,q} \le C h^{-2\gamma} \| v_h \|_{h,q},
\end{equation*}
for $\gamma \in (0,1)$.
Therefore, the first term $R_{1,h}^n$ is estimated as
\begin{align}
\| (-A_h)^{-\alpha} R_{1,h}^n \|_{h,q}
&= \| (-A_h)^{1-\alpha} (P_h - R_h) {U}^{n+1} \|_{h,q} \nonumber \\
&\le C h^{-2(1-\alpha)} \cdot h^2 \| {U}^{n+1} \|_{W^{2,q}}  \nonumber\\
& = C h^{2 \alpha} \| {U}^{n+1} \|_{W^{2,q}}.
\label{eq:R1h}
\end{align} 
Similarly, from \eqref{eq:interpolation-2} and Lemmas \ref{lem:projections} and \ref{lem:Kh}, we have
\begin{equation*}
\| (-A_h)^{-\alpha} (K_h^{-1} - I) P_h {V}^{n+1} \|_{h,q}
\le Ch^{2 \alpha} \| {V}^{n+1} \|_{W^{1,q}}.
\end{equation*}
Combining this inequality with Lemma \ref{lem:fdm}, we have
\begin{equation}
\left( \sum_{n=0}^{N-1} \| (-A_h)^{-\alpha} R_{2,h}^n \|_{h,q}^p \tau \right)^{1/p}
\le C(h^{2\alpha} + \tau).
\label{eq:R2h}
\end{equation} 
Since ${f}$ is globally Lipschitz continuous,
we have by \eqref{eq:stability_fractional}
\begin{multline}
 \left( \sum_{n=0}^{N-1} \| (-A_h)^{-\alpha} R_{3,h}^n \|_{h,q}^p \tau
 \right)^{1/p}\\
 \le CL \left[ 
\left( \sum_{n=0}^{N-1} \| (P_h - I){U}^n \|_{h,q}^p \tau \right)^{1/p}
+ \left( \sum_{n=0}^{N-1} \| {U}^{n+1} - {U}^n \|_{h,q}^p \tau \right)^{1/p} \right] \\
 \le CL (h^2 + \tau),
\label{eq:R3h} 
\end{multline}
where $L$ is the Lipschitz constant of ${f}$. 
The equations \eqref{eq:R1h}, \eqref{eq:R2h}, and \eqref{eq:R3h} yield
\begin{equation}
\left( \sum_{n=0}^{N-1} \| (-A_h)^{-\alpha} {r}_{2,h}^n \|_{h,q}^p \tau \right)^{1/p}
\le C(h^{2\alpha} + \tau).
\label{eq:r2h}
\end{equation}

Now, we are ready to show \eqref{eq:frac12}.  
We designate some constants appearing in this proof.
Since $A_h$ has discrete maximal regularity on $J_\infty$ in $X_{h,q}$ uniformly for $h$,
there exists $C_\mathrm{DMR} > 0$ depending only on $p$, $q$, $\Omega$ satisfying
\begin{equation}
\| v_h \|_{Y_S} 
\le C_\mathrm{DMR} \left( 
\| g_{h,1} \|_{l^p_\tau(N_S;X_{h,q})} + \| x_h \|_{1-1/p, p} \right) ,
\label{eq:dmr-Ah-rev}
\end{equation}
for every $g_h = (g_h^n)_n \in l^p(N_S; X_{h,q})$ and $x_h \in S_h$, 
where $v_h = (v_h^n)_n$ is the solution of
\begin{equation*}
\begin{cases}
(D_\tau v_h)^n = A_h v_h^{n+1} + g_h^{n+1}, & n = 0, \dots, N_S - 1, \\
v_h^0 = x_h.
\end{cases}
\end{equation*}
In view of \eqref{eq:stability_fractional} and the Lipschitz continuity
 of $f$, we have 
\begin{equation*}
C_\mathrm{Lip}  =
\sup \left\{ \frac{\| (-A_h)^{-\alpha} ({F}_h(v_h) - {F}_h(w_h) ) \|_{h,q}}{\| v_h - w_h \|_{h,q}}
\lmid  \begin{array}{l}
h>0, \ v_h, w_h \in S_h, \\ 
v_h\ne w_h
\end{array}
 \right\}<\infty,
\end{equation*}
which is the Lipschitz constant of $(-A_h)^{-\alpha}\circ {F}_h$.
Finally, we set 
\begin{equation}
C_0 = C_\mathrm{DMR} C_\mathrm{Lip} |\Omega|^{1/q}, 
\end{equation}
where $|\Omega|$ denotes the $d$-dimensional Lebesgue measure.

Let ${e}_{j,h}$ ($j=1,2$) be the solution of
\begin{equation}
\begin{cases}
(D_\tau {e}_{j,h})^n = A_h {e}_{j,h}^{n+1} + {r}_{j,h}^n, & n = 0, \dots, N_T - 1, \\
{e}_{j,h}^0 = 0.
\end{cases} \label{eq:e_j,h-rev}
\end{equation}
It is apparent that ${e}_h = {e}_{1,h} + {e}_{2,h}$.
Moreover, for every $S < T_\infty$, one can obtain
\begin{equation}
 \| (-A_h)^{-\alpha} {e}_{2,h} \|_{Y_S} 
\le C(h^{2\alpha} + \tau) 
\label{eq:e_2,h-rev}
\end{equation}
by \eqref{eq:stability_fractional} and \eqref{eq:r2h}.

Next, it is necessary to derive an estimation for ${e}_{1,h}$.
Take $S < T$ arbitrarily.
Since ${e}_{1,h}$ is the solution of \eqref{eq:e_j,h-rev},
discrete maximal regularity \eqref{eq:dmr-Ah-rev} and
Lemma \ref{lem:trace-fractional} yield
\begin{align*}
&\| (-A_h)^{-\alpha} {e}_{1,h} \|_{Y_S}  \\
& \le C_\mathrm{DMR} 
\left( \sum_{n=0}^{N_S - 1} \| {F}_h({u}_h^n) - {F}_h(P_h {U}^n) \|_{h,q}^p \tau \right)^{1/p} \\
& \le C_\mathrm{DMR} \left[
  C_\mathrm{Lip} \left( \sum_{n=0}^{N_S - 1} \| {e}_{1,h}^n \|_{h,q}^p \tau \right)^{1/p} 
+ C_\mathrm{Lip} \left( \sum_{n=0}^{N_S - 1} \| {e}_{2,h}^n \|_{h,q}^p \tau \right)^{1/p} \right] \\
& \le \Cdmr C_\mathrm{Lip} |\Omega|^{1/q} S^{1/p} \max_{0\le n \le N_S-1} \| {e}_{1,h}^n \|_{L^\infty} + C(h^{2\alpha} + \tau) \\
& \le C_0 S^{1/p} \| (-A_h)^{-\alpha} {e}_{1,h} \|_{Y_S} + C(h^{2\alpha} + \tau).
\end{align*}
Consequently, taking $S = (2C_0)^{-p}$, we obtain
\begin{equation}
\| (-A_h)^{-\alpha} {e}_{h,1} \|_{Y_{T_1}} \le C (  h^{2\alpha} + \tau )
\label{eq:local_error-rev0}
\end{equation}
with $T_1 = (2C_0)^{-p}$. This, together with \eqref{eq:e_2,h-rev}, implies \eqref{eq:frac12}.

\noindent \textbf{Step 2. }
We prove \eqref{eq:frac11} for any $T\in (0,\infty)$. 
We denote the constants appearing in Lemma \ref{prop:trace-finite} by $C_\mathrm{tr}$, and set
\begin{equation*}
C_1 = C_0 C_\mathrm{tr}, \quad C_2 = C_\mathrm{DMR} C_\mathrm{tr}.
\end{equation*}
Then we show that 
\begin{equation}
\| (-A_h)^{-\alpha} {e}_{1,h}^{\cdot + N_S} \|_{Y_\sigma} 
\le C \left( \| (-A_h)^{-\alpha} {e}_{1,h} \|_{Y_S} +  h^{2\alpha} + \tau \right)
\label{eq:extension-rev}
\end{equation}
for all $S < T$ and $\sigma \le \min\{T_1 ,\ T-S \}$.
Take $S < T$ and $\sigma \le \min\{T_1 ,\ T-S \}$ arbitrarily, and set ${w}_{j,h}^n = {e}_{j,h}^{n+N_S}$ ($j=1,2$).
Then, ${w}_{1,h}$ satisfies
\begin{equation*}
\begin{cases}
(D_\tau {w}_{1,h})^n = A_h {w}_{1,h}^{n+1} + {F}_h({u}_h^{n+N_S}) - {F}_h(P_h {U}^{n+N_S}), & n = 0, \dots, N_{T-S}, \\
{w}_{1,h}^0 = {e}_{1,h}^{N_S}.
\end{cases}
\end{equation*}
Therefore, discrete maximal regularity \eqref{eq:dmr-Ah-rev}, Lemmas \ref{prop:trace-finite}, \ref{lem:trace-fractional}, and \eqref{eq:local_error-rev0} yield
\begin{align*}
&\| (-A_h)^{-\alpha} w_{1,h} \|_{Y_\sigma} \\
& \le C_\mathrm{DMR} \Bigg[ 
\left( \sum_{n=0}^{N_\sigma - 1} \left\| (-A_h)^{-\alpha}\left[{F}_h({u}_h^{n+N_S}) - {F}_h(P_h {U}^{n+N_S})\right] \right\|_{h,q}^p \tau \right)^{1/p} \\
& \hspace{45ex} 
+ \| (-A_h)^{-\alpha} {e}_{1,h}^{N_S} \|_{1-1/p, p} \Bigg] \\
& \le C_\mathrm{DMR} \Bigg[
  C_\mathrm{Lip} \left( \sum_{n=0}^{N_\sigma - 1} \| {w}_{1,h}^n \|_{h,q}^p \tau \right)^{1/p} 
+ C_\mathrm{Lip} \left( \sum_{n=0}^{N_\sigma - 1} \| {w}_{2,h}^n \|_{h,q}^p \tau \right)^{1/p} \\
& \hspace{45ex} 
+ C_\mathrm{tr} \| (-A_h)^{-\alpha} {e}_{1,h} \|_{Y_S} \Bigg] \\
& \le C_0 \sigma^{1/p} 
      \left( \| (-A_h)^{-\alpha} {w}_{1,h} \|_{Y_\sigma} + \| (-A_h)^{-\alpha} {e}_{1,h}^{N_S} \|_{1-1/p,p} \right)
    + C(h^{2\alpha} + \tau) \\
    & \hspace{45ex} 
    + C_\mathrm{DMR} C_\mathrm{tr} \| (-A_h)^{-\alpha} e_{1,h} \|_{Y_S} \\
& \le \frac{1}{2} \| (-A_h)^{-\alpha} {w}_{1,h} \|_{Y_\sigma} 
    + \left(\frac{C_\mathrm{tr}}{2} + C_2 \right) \| (-A_h)^{-\alpha} {e}_{1,h} \|_{Y_S}
    +C(h^{2\alpha} + \tau),
\end{align*}
since $\sigma \le T_1 = (2C_0)^{-p}$.
Therefore, we obtain \eqref{eq:extension-rev}.

Noting that $N_{S+\sigma} \le N_S + N_\sigma$, one obtains
\begin{equation*}
\| v_h \|_{Y_{S+\sigma}} \le \| v_h \|_{Y_S} + \| v_h^{\cdot + N_S} \|_{Y_\sigma}
\end{equation*}
for $v_h \in l^p(N_S + N_\sigma;S_h)$ and $S, \sigma >0$.
Therefore, we can inductively establish \eqref{eq:frac11} 
from \eqref{eq:local_error-rev0} and \eqref{eq:extension-rev}.
Now we can complete the proof owing to Lemma \ref{lem:trace-fractional}.
\qed\end{proof}

\medskip

Finally, we state the following proof.

 \begin{proof}[Proof of Theorems \ref{thm:sl} and \ref{thm:infty}]
Observe that
\[
 \|u^n_h-U^n\|_{L^q}\le \|e_h^n\|_{L^q}+\|P_h U^n-U^n\|_{L^q}\le\|e_h^n\|_{L^q}+Ch^2\|U^n\|_{W^{2,q}}
\]  
by Lemma \ref{lem:projections}.
Therefore, it suffices to prove
\[
\left( \sum_{n=1}^{N_T} \| {e}_h^n \|_{L^q}^p \tau \right)^{1/p}
\le C(h^2 + \tau)\quad\mbox{and}\quad 
\max_{0 \le n \le N_T} \| {e}_h^n \|_{L^\infty}
\le C(h^{2\alpha} + \tau)
\]
for $\alpha\in (0,\alpha_{p,q,d})$. 
To this end, let
\begin{equation*}
M= \| u \|_{L^\infty(\Omega \times (0, T))} 
   + \sup_{h > 0} \| P_h u \|_{L^\infty(\Omega \times (0, T))}
\end{equation*}
for the solution $u$ of \eqref{eq:semilinear} and $T \in (0,T_\infty)$.
It is apparent that $M$ is finite since the $L^2$-projection $P_h$ is
 stable in the $L^\infty$-norm (Lemma \ref{lem:projections}). We
 introduce 
\begin{equation*}
\tilde{f}(z) = \tilde{f}_M(z) = 
\begin{cases}
f(z), & |z| \le M, \\
f\left( M \frac{z}{|z|} \right), & |z| > M.
\end{cases}
\end{equation*}
Then, $\tilde{f}$ is a globally Lipschitz continuous function. 
We consider the problems \eqref{eq:semilinear} and
 \eqref{eq:approx-semilinear} with replacement of $f$ by $\tilde{f}$, 
and denote the corresponding solutions by $\tilde{u}$ and $\tilde{u}_h$,
 respectively. Moreover, we consider the error $\tilde{e}_h =
 (\tilde{e}_h^n)_{n=0}^{N_T}\in S_h^{N_T+1}$, where $\tilde{e}_h^n =
 \tilde{u}_h^n - P_h \tilde{u}(n\tau)$. 
 
In view of Lemma \ref{thm:global_error-truncated},
the following error estimate holds:
\begin{equation*}
\| (-A_h)^{-\alpha}\tilde{e}_h \|_{Y_{T}}
\le C(h^{2\alpha} + \tau)
\end{equation*}
for any $\alpha\in [0,1]$. By setting $\alpha=1$, we obtain
\begin{equation}
\label{eq:e100}
\left( \sum_{n=1}^{N_T} \| \tilde{e}_h^n \|_{L^q}^p \tau \right)^{1/p}
\le C(h^2 + \tau).
\end{equation}
Applying Lemma \ref{lem:trace-fractional}, we can deduce
\begin{equation}
\label{eq:e101}
\max_{0 \le n \le N_T} \| \tilde{e}_h^n \|_{L^\infty}
\le C(h^{2\alpha} + \tau),
\end{equation}
for $\alpha\in (0,\alpha_{p,q,d})$. 
 
At this stage, we have $\tilde{u} = u$ by the unique solvability of \eqref{eq:semilinear}. 
Indeed, $\| u \|_{L^\infty(\Omega \times (0, T))}
 \le M$ implies $\tilde{f}(u(x,t)) = f(u(x,t))$ for every $(x,t) \in
 \Omega \times (0,T)$. 
Moreover, according to \eqref{eq:e101}, we estimate as 
\begin{align*}
\max_{0 \le n \le N_T} \| \tilde{u}_h^n \|_{L^\infty} 
&\le \max_{0 \le n \le N_T} \| \tilde{e}_h^n \|_{L^\infty} 
  + \max_{0 \le n \le N_T} \| P_h U^n \|_{L^\infty} \\
&\le C(h^{2\alpha} + \tau) + \sup_{h > 0} \| P_h u
 \|_{L^\infty(\Omega 
\times (0, T))}
\end{align*}
for $\alpha \in (0, \alpha_{p,q,d})$.
Therefore, there exist $h_0 > 0$ and $\tau_0 > 0$ such that
\begin{equation*}
\max_{0 \le n \le N_T} \| \tilde{u}_h^n \|_{L^\infty} \le M,
\qquad \forall h \le h_0, \quad \forall \tau \le \tau_0,
\end{equation*}
which implies that $\tilde{f}(\tilde{u}_h^n) = f(\tilde{u}_h)$.
Again, the unique solvability of \eqref{eq:approx-semilinear} yields
 $\tilde{u}_h = u_h$ for $h \le h_0$ and $\tau \le \tau_0$.
 Hence we can replace $\tilde{e}_h^n$ by $e_h^n$ in \eqref{eq:e100} and
 \eqref{eq:e101}, which completes the proof of Theorems \ref{thm:sl} and
 \ref{thm:infty}. 
\qed\end{proof}

\begin{remark}
\label{rem:bound}
Based on the same assumptions of Lemma \ref{thm:global_error-truncated},
the solution $u_h = (u_h^n)_n$ of \eqref{eq:approx-semilinear} admits
\begin{equation*}
\| u_h^n \|_{L^\infty} 
\le C \| u_0 \|_{L^\infty} e^{TL}. 
\end{equation*}
We briefly show this inequality. 
Let $T_{h, \tau} = (I - \tau A_h)^{-1}$ and $F_h = K_h^{-1} \circ P_h \circ f$.
Then
the first equation of \eqref{eq:approx-semilinear} is equivalent to
\begin{equation*}
u_h^n = T_{h, \tau}^n P_h u_0 
+ \tau \sum_{n=0}^{n-1} T_{h, \tau}^{n-j} F_h(u_h^j) 
\end{equation*}
for $n \in \bN$.
It follows from Lemma \ref{prop:disc-Laplacian2} and the Hille-Yosida theorem that
\[
 \| \lambda^n R(\lambda; A_h)^n \|_{\cL(X_{h,\infty})} \le 1
\]
for all $n \in \bN$ and $\lambda > 0$.
Particularly, we have
\begin{equation*}
\| T_{h, \tau}^n \|_{\cL(X_{h,\infty})} \le 1 \label{eq:disc-semigroup}
\end{equation*}
for all $n \in \bN$.
Moreover, one can find $L > 0$, independent of $h$, such that
\begin{equation*}
\| F_h(v_h) - F_h(w_h) \|_{L^\infty} \le L \| v_h - w_h \|_{L^\infty},
\quad \forall h>0
\end{equation*}
for $v_h, w_h \in S_h$ by the globally Lipschitz continuity of $f$ 
and Lemmas \ref{lem:projections} and \ref{lem:lumping}.
Then, we obtain
\begin{equation*}
\| u_h^n \|_{L^\infty} 
\le C \| u_0 \|_{L^\infty} + \tau \sum_{j=0}^{n-1} L \| u_h^j \|_{L^\infty}.
\end{equation*}
Therefore, the well-known discrete Gronwall lemma \cite[Lemma 2.3]{Qua14} implies
\begin{equation*}
\| u_h^n \|_{L^\infty} 
\le C \| u_0 \|_{L^\infty} e^{n \tau L}
\le C \| u_0 \|_{L^\infty} e^{T L}
\end{equation*}
for $n\in\mathbb{N}$. 
\end{remark}

\appendix
\section{$H^\infty$-functional calculus}
\label{section:functional_calculus}

In this appendix, we review the notion of $H^\infty$-functional calculus.
We present only the definition and the theorem used for this study.
For relevant details, one can refer to \cite{CowDMY96} and references therein.
Throughout this section, $X$ denotes a Banach space and $\Sigma_\omega$
is the sector defined as \eqref{eq:sec}.

\begin{definition}
For $\omega \in (0, \pi)$, a linear operator $A$ is \emph{of type
 ${\omega}$} if and only if
\begin{enumerate}
\item $A$ is closed and densely defined,
\item $\sigma(A) \subset \overline{\Sigma}_\omega$,
\item for each $\theta \in (\omega, \pi]$, there exists $C_\theta > 0$
      satisfying $\| R(z; A) \|_{\cL(X)} \le {C_\theta}/{|z|}$ 
for all $z \in \bC \setminus \Sigma_\theta$ with $z \ne 0$.
\end{enumerate}
\end{definition}

Every positive type operator is of type $\omega$ for some $\omega \in (0,\pi/2)$.
Now, we define the functions of operators of type $\omega$.
For $\theta \in (0, \pi)$, we set
\begin{equation*}
\Psi(\Sigma_\theta) = \bigcup_{C \ge 0,\ s > 0}
\left\{ f \in H^\infty(\Sigma_\theta) \lmid
|f(z)| \le C \frac{|z|^s}{1 + |z|^{2s}}, \quad \forall z \in \Sigma_\theta \right\},
\end{equation*}
where $H^\infty(\Sigma_\theta)$ is defined as \eqref{eq:H-infty}.
Let $\Gamma_\vartheta = \{ - t e^{-i \vartheta} \mid - \infty < t < 0 \} 
\cup \{ t e^{i \vartheta} \mid 0 \le t < \infty \}$ be a contour
for $\vartheta \in (0, \pi)$, which is oriented so that the imaginary parts increase along $\Gamma_\vartheta$.

\begin{definition}
Let $A \in \cP(X;K)$ for some $K \ge 1$.
Assume that $A$ is of type $\omega$ and let $\omega < \vartheta < \theta$.
Then, we define the function of operator $A$ as
\begin{equation*}
\psi(A) = \frac{1}{2 \pi i} \int_{\Gamma_\vartheta} (A - zI)^{-1} \psi(z) dz
\end{equation*}
for $\psi \in \Psi(\Sigma_\theta)$.
We also define $m(A)$ for $m \in H^\infty(\Sigma_\theta)$ as
\begin{equation*}
m(A) = \psi_0(A)^{-1}(\psi_0 m)(A),
\end{equation*}
where $\psi_0(z) = z/(1+z)^2$.
\end{definition}

In the case in which $X$ is a Hilbert space and the operator $A$ is positive type and self-adjoint,
we can define $m(A)$ for $m \in L^\infty(\bR^+)$ by the spectral decomposition.
It is natural to wonder whether these two definitions coincide. The answer is as follows.
See for example \cite[Theorem 4.6.7 in Chapter III]{Ama95} for the proof.

\begin{lemma}\label{thm:functional_Hilbert}
Let $X$ be a Hilbert space and $A \in \cP(X)$.
Assume that $A$ is self-adjoint and let $E_A(\lambda)$ be its spectral
 decomposition.
 Then, we have
\begin{equation*}
m(A) = \int_0^\infty m(\lambda) dE_A(\lambda)
\end{equation*}
for $m \in H^\infty(\Sigma_\theta)$.
\end{lemma}

\section{Remark on the scheme (17)} 
\label{sec:scheme}

An alternate of the scheme \eqref{eq:disc-heat} is given as 
\begin{equation}
 (D_\tau u_h )^n = A_h u_h^{n+\theta} + P_h G^{n+\theta},
\label{eq:disc-heat2}
\end{equation}
or, equivalently,
\begin{equation*}
((D_\tau u_h)^n, v_h)_h = -(\nabla u_h^{n+\theta}, \nabla v_h)_{L^2} 
+ (P_h G^{n+\theta}, v_h)_h, \quad \forall v_h \in S_h.
\end{equation*}
If taking \eqref{eq:disc-heat2} 
instead of the first equation of \eqref{eq:disc-heat}, 
we can only obtain the following error estimate: 
\begin{equation}
\left( \sum_{n=0}^{N_T-1} \| u_h^{n+\theta} - U^{n+\theta} \|_{L^q}^p \tau \right)^{1/p} \le C( h + \tau^{j_\theta}),
\label{eq:error_estimate2}
\end{equation}
since Lemma \ref{lem:Kh} is not available.
This shortcoming is confirmed by numerical examples as
follows.

Let us consider the following two-dimensional heat equation in $\Omega = (0,1)^2$:
\begin{equation}
\begin{cases}
\dfrac{\partial u}{\partial t}(x,y,t) = \Delta u(x,y,t) + g(x,y,t),
 & (x,y) \in \Omega, ~ 0<t\le T, \\
u(x,y,t) = 0, & (x,y) \in \partial\Omega, ~ 0<t\le T, \\
u(x,y,0) = x^{5/2} (1-x)^{5/2} y (1-y)
& (x,y) \in \Omega,
\end{cases}
\label{eq:example}
\end{equation}
where $T>0$ and
\begin{equation}
g(x,y,t) = x^{1/2} (1-x)^{1/2} e^t \left[ 
x^2(1-x)^2 y(1-y) - \frac{5}{4}(3-4x)(1-4x)y(1-y) + 2
 \right].
\end{equation}
The exact solution is $u(x,y,t) = x^{5/2} (1-x)^{5/2} y (1-y) e^{t}$.
We approximate the equation \eqref{eq:example} by the schemes \eqref{eq:disc-heat} and \eqref{eq:disc-heat2} with meshes such as Figure \ref{fig:meshes}, 
which satisfies the conditions (H1) and (H2).

\begin{figure}
\centering
\includegraphics[width = 0.3\textwidth]{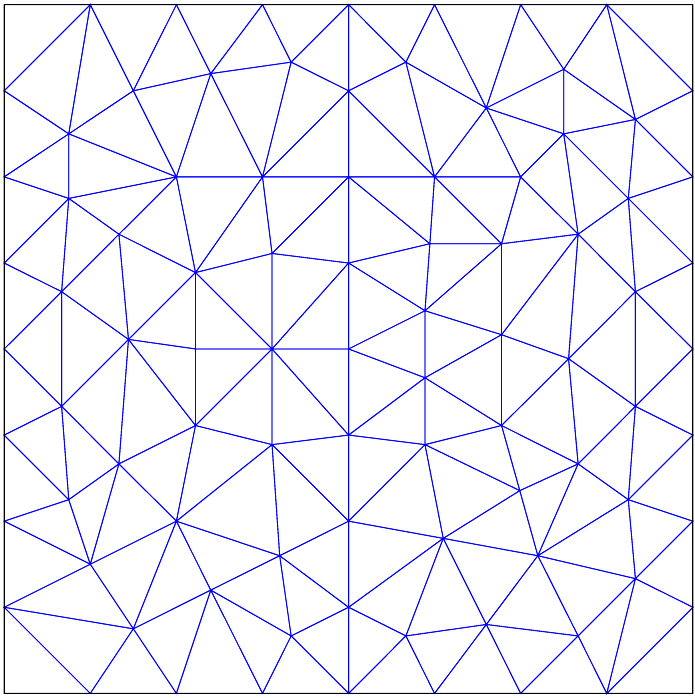}
\hspace{0.02\textwidth}
\includegraphics[width = 0.3\textwidth]{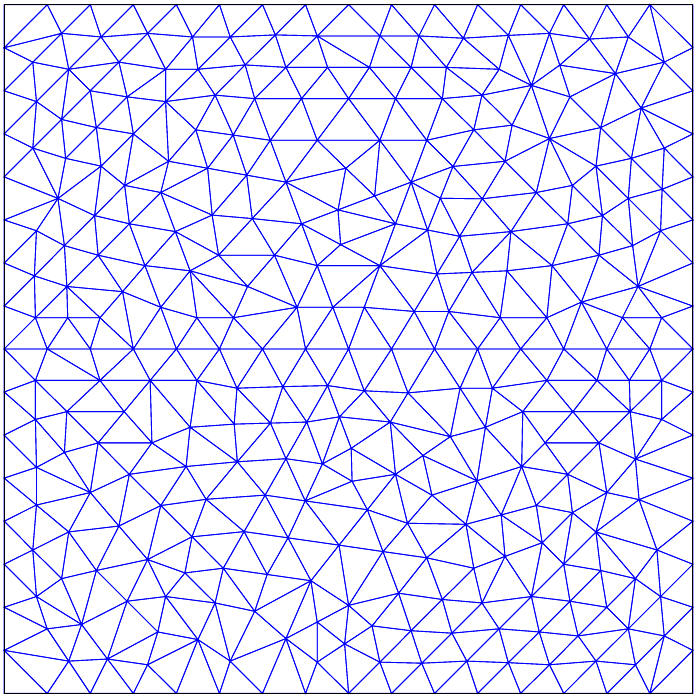}
\hspace{0.02\textwidth}
\includegraphics[width = 0.3\textwidth]{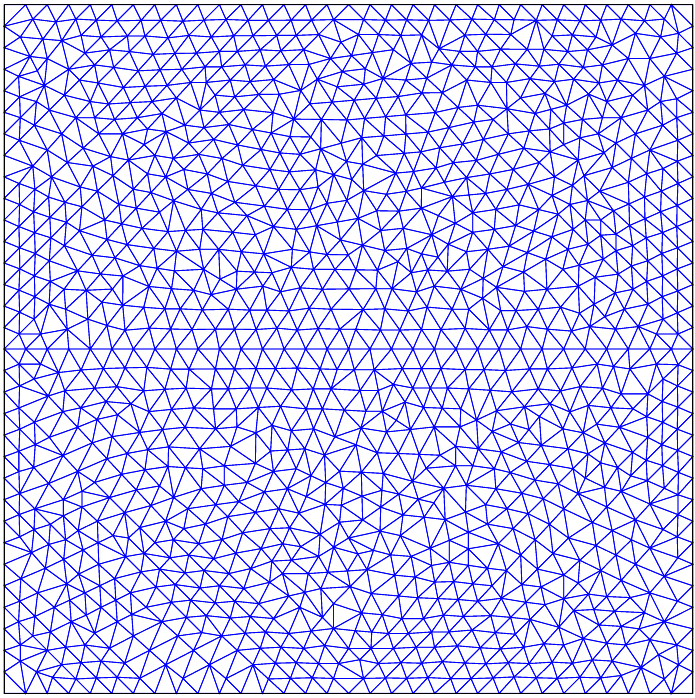}
\caption{Meshes.}
\label{fig:meshes}
\end{figure}

We consider the case for $\theta = 0$, $1/2$ and $1$.
When $\theta = 1/2$ and $\theta = 1$, we take $\tau$ as $\tau = h$ or $\tau = h^2$.
In the case for $\theta = 0$, $\tau$ should be chosen to satisfy the condition \eqref{eq:sufficient-NR}.
We take $\varepsilon = \sin\theta_q$ and 
\begin{equation}
\tau = \frac{\sin \theta_q}{(1 -2 \theta)(d+1)^2} \kappa_h^2 ,
\end{equation}
so that $\tau$ satisfies $\tau = O(h^2)$ by the inverse assumption.
We set the parameters as follows:
\begin{itemize}
\item $(p,q) = (4,2)$,
\item $T = 0.1$ ($\theta = 0$) or $T = 0.5$ ($\theta = 1/2, 1$).
\end{itemize}

Behavior of the errors is shown in Figure~\ref{fig:results}.
In these figures, cases 1--5 mean the following cases:
\begin{equation}
\begin{array}{ll}
\text{\textbf{case 1}: $\theta = 0$ ($\tau = O(h^2)$),} & \\[1ex]
\text{\textbf{case 2}: $\theta = 1/2$, $\tau = h$,} &
\text{\textbf{case 3}: $\theta = 1/2$, $\tau = h^2$,} \\[1ex]
\text{\textbf{case 4}: $\theta = 1$, $\tau = h$,} &
\text{\textbf{case 5}: $\theta = 1$, $\tau = h^2$.}
\end{array}
\end{equation}
Let us consider the order of the error.
In case 4 with the scheme \eqref{eq:disc-heat}, for example, from Theorem \ref{thm:error_linear} and $\tau = h$, we have
\begin{equation}
(\text{The error}) \le C(h^2 + \tau) \le Ch
\end{equation}
if $h$ is sufficiently small.
We summarize these theoretical orders and results in Table~\ref{tab:order}.
When we use the scheme \eqref{eq:disc-heat}, the orders correspond to the theoretical bounds.
In the case for the scheme \eqref{eq:disc-heat2}, all orders are expected to be $O(h)$.
However, except for case 4, the orders are apparently $O(h^2)$.
It is of course no problem since the error estimate \eqref{eq:error_estimate2} is just an upper bound.
In case 4, it also seems that the order is $O(h^2)$.
However, when we compute \eqref{eq:disc-heat2} in case 4 for smaller $h$, 
the error decreases more slowly. It seems to approach $O(h^\alpha)$ for some $\alpha \in [1,2)$: Figure~\ref{fig:results_case4}.
We leave more rigorous error estimates for the scheme \eqref{eq:disc-heat2} as a subject for future work.

\begin{figure}
\centering
\begin{minipage}{.49\textwidth}
\centering\subcaptionbox{Scheme \eqref{eq:disc-heat} \label{fig:result1}}{
\includegraphics[width=0.95\textwidth]{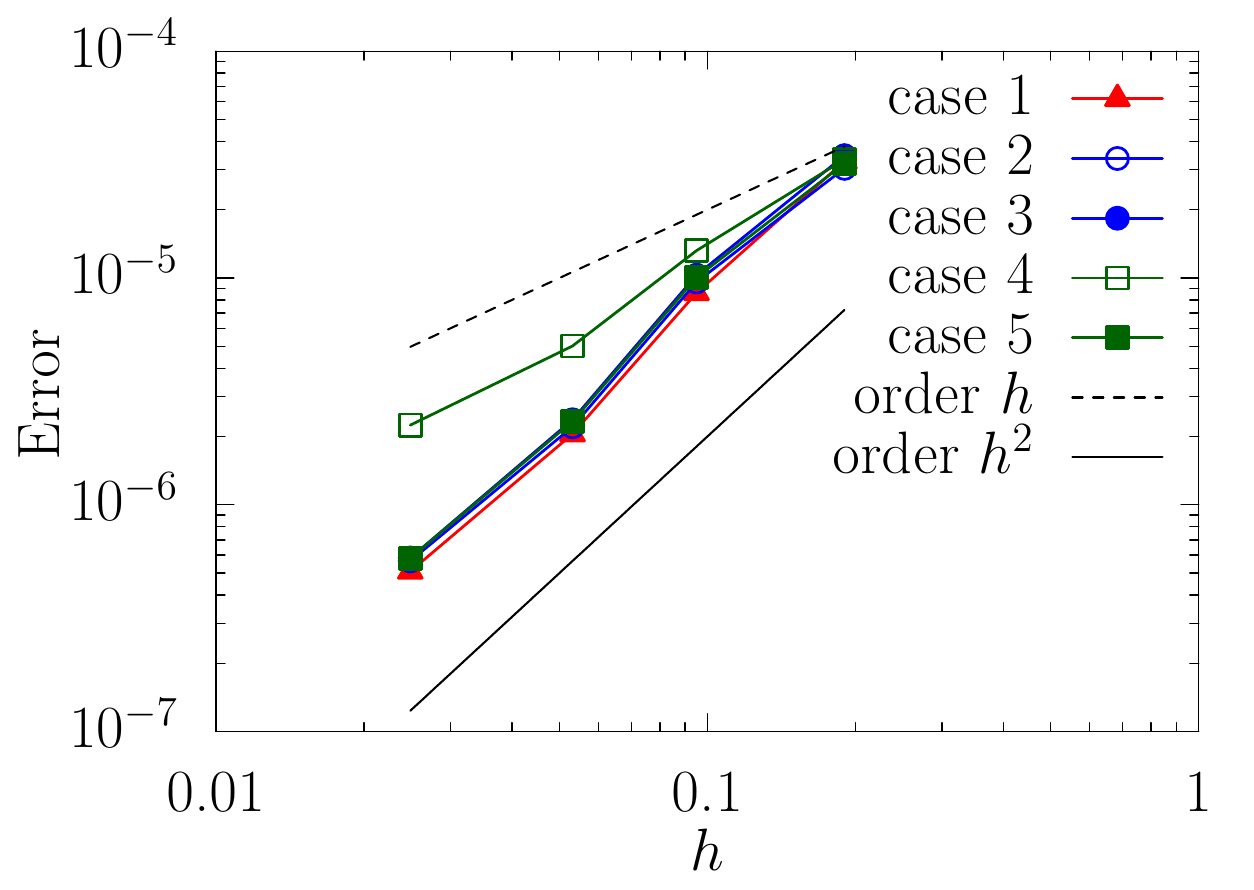}
}
\end{minipage}
\begin{minipage}{.49\textwidth}
\centering\subcaptionbox{Scheme \eqref{eq:disc-heat2} \label{fig:result2}}{
\includegraphics[width=0.95\textwidth]{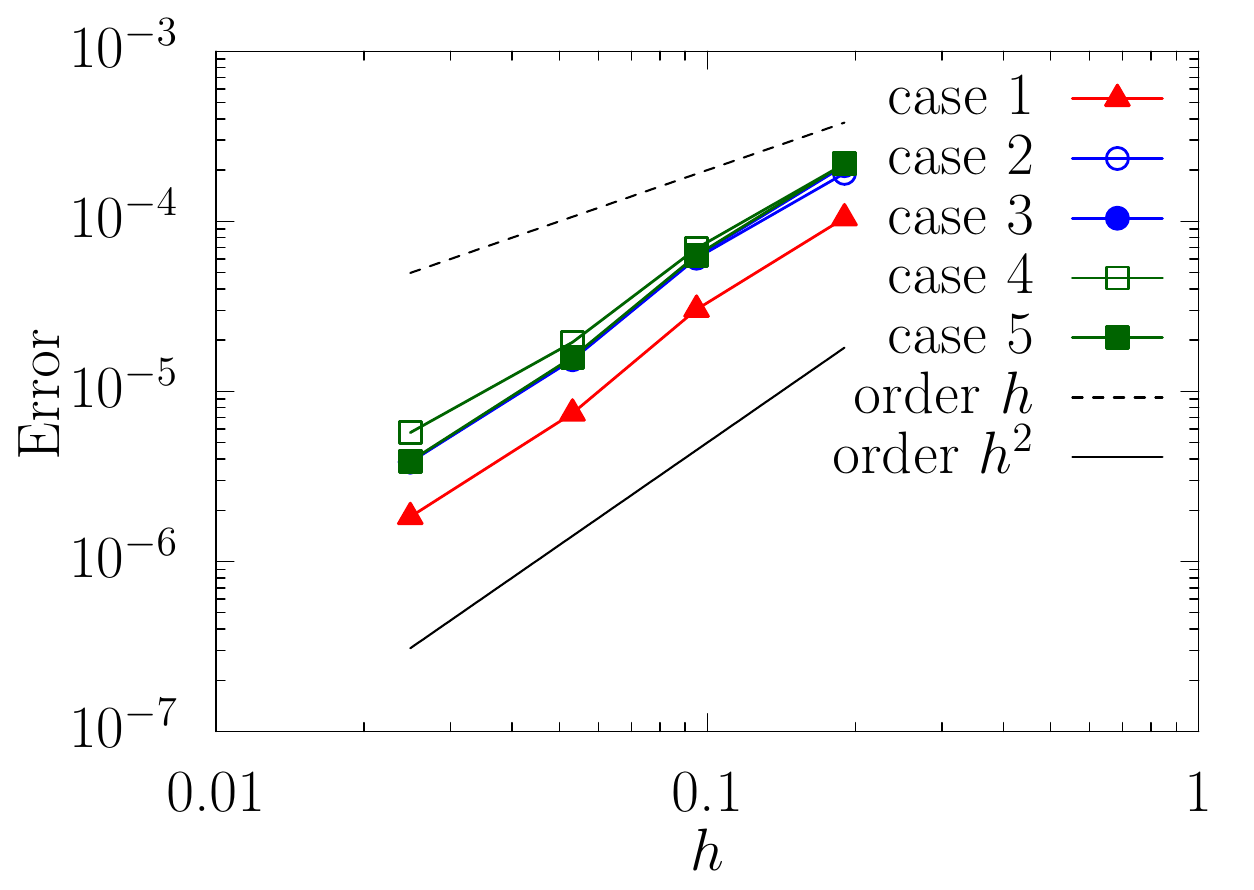}
}
\end{minipage}
\caption{Behavior of $L^4$-$L^2$-errors.}
\label{fig:results}
\end{figure}

\begin{table}
\centering
\begin{minipage}{.49\textwidth}
\centering
\subcaptionbox{Scheme \eqref{eq:disc-heat} \label{tab:order1}}{
\begin{tabular}{|c|c|c|c|}\hline
\multicolumn{2}{|c|}{conditions} & bounds & results \\ \hline
$\theta = 0$ & $\tau \propto h^2$  & $O(h^2)$     & $O(h^2)$ \\ \hline
\multirow{2}{*}{\vspace*{-1.5ex}$\theta = 1/2$} & $\tau = h$   & $O(h^2)$ & $O(h^2)$ \\ \cline{2-4}
                                & $\tau = h^2$ & $O(h^2)$ & $O(h^2)$ \\ \hline
\multirow{2}{*}{\vspace*{-1.5ex}$\theta = 1$}   & $\tau = h$   & $O(h)$   & $O(h)$ \\ \cline{2-4}
                                & $\tau = h^2$ & $O(h^2)$ & $O(h^2)$ \\ \hline
\end{tabular}
}
\end{minipage}
\begin{minipage}{.49\textwidth}
\centering
\subcaptionbox{Scheme \eqref{eq:disc-heat2} \label{tab:order2}}{
\begin{tabular}{|c|c|c|c|} \hline
\multicolumn{2}{|c|}{conditions} & bounds & results \\ \hline
$\theta = 0$ & $\tau \propto h^2$  & $O(h)$     & $O(h^2)$ \\ \hline
\multirow{2}{*}{\vspace*{-1.5ex}$\theta = 1/2$} & $\tau = h$   & $O(h)$ & $O(h^2)$ \\ \cline{2-4}
                                & $\tau = h^2$ & $O(h)$ & $O(h^2)$ \\ \hline
\multirow{2}{*}{\vspace*{-1.5ex}$\theta = 1$}   & $\tau = h$   & $O(h)$   & ? \\ \cline{2-4}
                                & $\tau = h^2$ & $O(h)$ & $O(h^2)$ \\ \hline
\end{tabular}
}
\end{minipage}
\caption{The convergence rates: theoretical bounds and results.}
\label{tab:order}
\end{table}

\begin{figure}[t]
\centering
\includegraphics[width = 0.4655\textwidth]{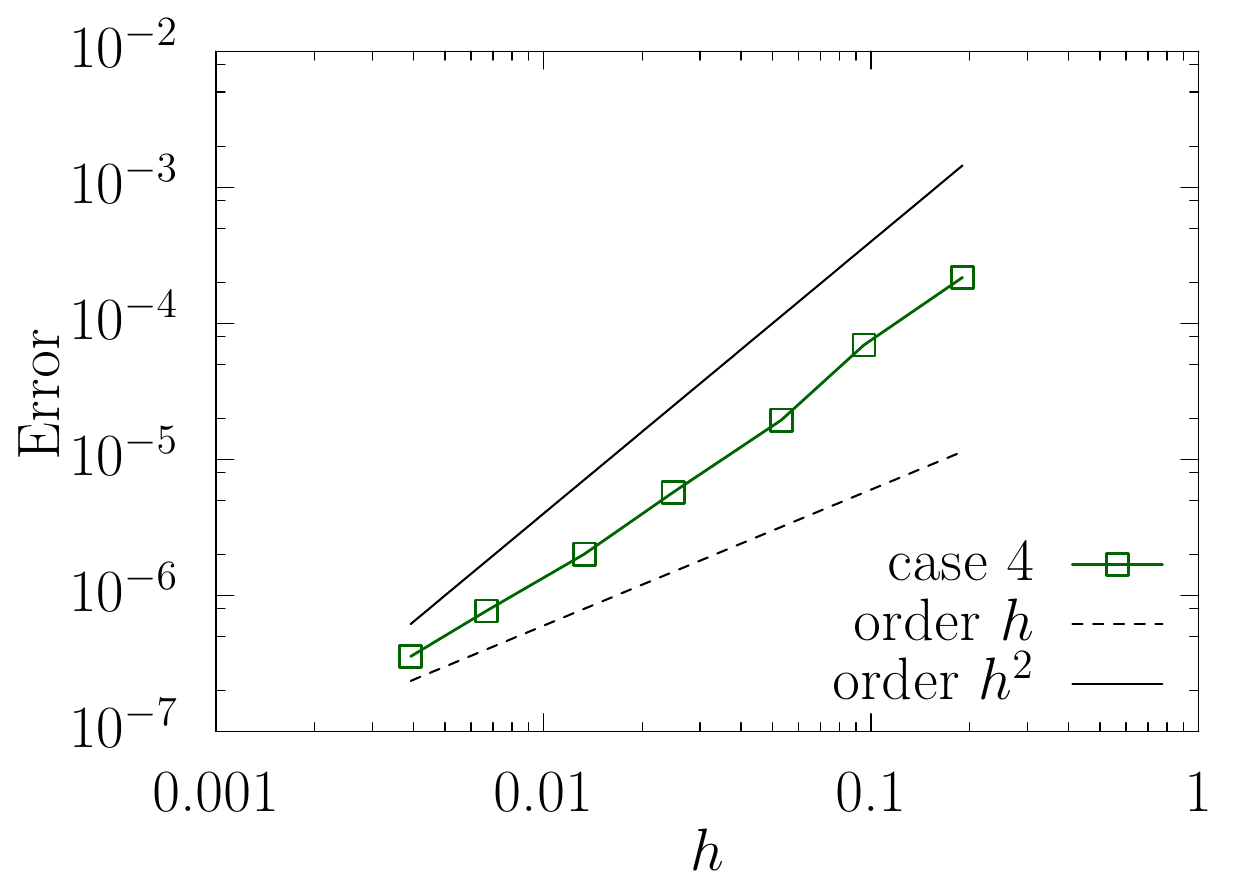}
\caption{Behavior of $L^4$-$L^2$-errors in case 4 with scheme
(17) 
for smaller $h$.}
\label{fig:results_case4}
\end{figure}

\begin{acknowledgements}
The first author was supported by the Program for Leading Graduate Schools, MEXT,
Japan and JSPS KAKENHI (15J07471).
The second author was supported by JST, CREST, and JSPS KAKENHI (15H03635, 15K13454). 
\end{acknowledgements}

\bibliographystyle{plain}

\end{document}